\documentclass[11pt,reqno]{amsart}
\usepackage{amsmath,mathtools}
\usepackage{mathrsfs}
\usepackage{graphicx}
\usepackage{epsfig,epsf,psfrag}
\usepackage{amsfonts}
\usepackage{setspace}
\usepackage{color}
\usepackage[font=small]{caption}
\usepackage[font=footnotesize]{subcaption}
\usepackage[top=1in, bottom=1in, left=1.2in, right=1.2in]{geometry}
\usepackage{afterpage}
\usepackage{bm}
\usepackage{multicol}
\usepackage{multirow}
\usepackage{diagbox}
\usepackage{booktabs}
\usepackage{algorithm}
\usepackage{algorithmic}
\usepackage{lineno}
\usepackage{indentfirst}
\usepackage{cases}

\usepackage{isomath}

\newtheorem{theorem}{Theorem}[section]
\newtheorem{definition}[theorem]{Definition}
\newtheorem{corollary}[theorem]{Corollary}

\newtheorem{remark}[theorem]{Remark}

\newcommand{\bgamma}{\boldsymbol{\gamma}}

\newcommand{\bkappa}{\boldsymbol{\kappa}}

\newcommand{\btau}{\boldsymbol{\tau}}

\newcommand{\bchi}{\boldsymbol{\chi}}

\newcommand{\bomega}{\boldsymbol{\omega}}

\newcommand{\bGamma}{\boldsymbol{\Gamma}}

\newcommand{\bOmega}{\boldsymbol{\Omega}}
\allowdisplaybreaks

\title[]{Computation of shape Taylor expansions} 

\numberwithin{equation}{section}

\author{Gang Bao}
\address{School of Mathematical Sciences, Zhejiang University,
	Hangzhou, Zhejiang 310027, China}
\email{baog@zju.edu.cn}

\author{Jun Lai}
\address{School of Mathematical Sciences, Zhejiang University,
	Hangzhou, Zhejiang 310027, China}
\email{laijun6@zju.edu.cn}

\author{Haoran Ma}
\address{School of Mathematical Sciences, Zhejiang University,
	Hangzhou, Zhejiang 310027, China}
\email{MaHaoran77@zju.edu.cn}



\begin{document}

\begin{abstract}
 Shape derivative is an important analytical tool for studying scattering problems involving perturbations in scatterers. Many applications, including inverse scattering, optimal design, and uncertainty quantification, are based on shape derivatives. However, computing high order shape derivatives is challenging due to the complexity of shape calculus. This work introduces a comprehensive method for computing shape Taylor expansions in two dimensions using recurrence formulas. The approach is developed under sound-soft, sound-hard, impedance, and transmission boundary conditions. Additionally, we apply the shape Taylor expansion to uncertainty quantification in wave scattering, enabling high order moment estimation for the scattered field under random boundary perturbations. Numerical examples are provided to illustrate the effectiveness of the shape Taylor expansion in achieving high order approximations.
\end{abstract}

\maketitle

\section{Introduction}\label{Introduction}

Shape derivative in scattering problems establishes a connection between variations in the scattered field and the perturbations in the shape of a scatterer. Over the past few decades, it has played an important role in various areas such as inverse scattering~\cite{bao2015inverse,bao2011imaging,fink2023domain,hagemann2019solving}, optimal design~\cite{bao2014optimal,ma2024inverse}, and uncertainty quantification~\cite{escapil2023shape,harbrecht2013first}. In this work, we consider a scattered field $u$ that depends smoothly on the perturbation velocity $\mathbf{v}$ of a scatterer $D$. Our focus is on the computation of the shape Taylor expansion of arbitrary order $N>0$, defined as: 
\begin{eqnarray}\label{shapetaylor}
    \mathbf{Taylor}(\mathbf{x};u,\mathbf{v},N):= u(\mathbf{x}) + \epsilon\delta_{\mathbf{v}}u(\mathbf{x}) + \frac{\epsilon^2}{2!}\delta_{[\mathbf{v},2]}u(\mathbf{x}) + ... + \frac{\epsilon^N}{N!}\delta_{[\mathbf{v},N]}u(\mathbf{x}).
\end{eqnarray}
where $\delta_{[\mathbf{v},N]}u$ denotes the $N$th order shape derivative of $u$. 

While the theory of first order shape derivative has been well established, there are only a few works on the higher order shape derivatives~\cite{hiptmair2018shape}. The main reason is due to the complex and tedious derivation process of the higher order shape calculus~\cite{dolz2020higher}. Most of the existing works are problem specific~\cite{hagemann2020application, hettlich1999second, novotny2019topological3} and  focus only on deriving derivatives up to the second order. For instance,~\cite{hagemann2020application} made use of the second order shape derivative to solve the inverse scattering problem of reconstructing a perfect electric conductor (PEC) from the far field pattern, and~\cite{dolz2020higher} employed the second order expansion to estimate the scattered field in electromagnetics with random shape variations. However, for many shape reconstruction and  optimization problems, higher order shape derivatives are required for more accurate approximations. 

 Since the underlying differential equation that governs the shape derivatives of wave scattering is the same, deriving shape derivatives of different orders reduces to determining the appropriate boundary conditions.  Motivated by~\cite{hiptmair2018shape}, in our recent work~\cite{bao2025shape}, we have successfully derived these boundary conditions for shape derivatives of arbitrary order and constructed an explicit formula for the shape Taylor expansion \eqref{shapetaylor}.  The derivation employed tools from exterior differential forms, Lie derivatives, and material derivatives. The work establishes a unified framework for computing the high order shape perturbations in scattering problems, and the recurrence formulas are applicable to both acoustic and electromagnetic scattering models under a variety of boundary conditions. In particular, we have shown the boundary conditions for shape derivatives of order $N$ can be recursively obtained from those of order $N-1$.  However, an effective numerical method for evaluating these high order shape derivatives remains to be developed.

In this work, based on the theoretical framework established in \cite{bao2025shape}, we propose an effective numerical computational method for evaluating shape derivatives of arbitrary order in two-dimensional acoustic scattering problems, under sound-soft, sound-hard, impedance, and transmission boundary conditions. The numerical method is grounded in potential theory and boundary integral equations, enabling the construction of shape Taylor expansion \eqref{shapetaylor} numerically.  Using this numerical tool, we apply the shape Taylor expansion to uncertainty quantification in wave scattering problems, especially for estimating the statistical moments of the scattered field under random boundary perturbations~\cite{dolz2020higher,harbrecht2018second,henriquez2021shape}. It is worth mentioning that the first order shape expansion is commonly employed in the uncertainty quantification of small domain perturbations~\cite{hao2018computation,harbrecht2013first}. However, when the perturbation is slightly larger, the first order expansion becomes less effective~\cite{dolz2020higher}. We will show that the high order shape Taylor expansion can significantly improve the accuracy of the estimations.

The remainder of this paper is organized as follows. In Section~\ref{Shapecalculus}, we introduce the basic wave scattering formulation and recall the recurrence formulas for the shape derivatives of arbitrary order from \cite{bao2025shape}. Section~\ref{ShapeTaylor} details the numerical implementation of shape Taylor expansion for two-dimensional scattering problems. In particular, we present the explicit forms of the recurrence relations under four different boundary conditions, and then derive a recurrence formula for computing normal derivatives of arbitrary order of the scattered field. Section~\ref{Randomproblem} presents the main results for the moment estimation of the scattered field under random boundary perturbations. Section~\ref{Numerical} gives numerical examples to demonstrate the high order approximation property of the shape Taylor expansion. Finally, Section~\ref{Conclusion} concludes the paper.

\section{Recurrences of shape derivatives}\label{Shapecalculus}
Consider a bounded, simply connected open set $D\subset\mathbb{R}^2$ as the domain of an acoustic scatterer, with a smooth boundary $\bGamma:=\partial D$. Suppose the scatterer is illuminated by a time-harmonic incident wave $\phi(\mathbf{x})$. In the case of impenetrable boundaries, including sound-soft, sound-hard, and impedance boundaries, the scattered field $u$ is defined in the exterior domain $ D_{\rm ex}:=\mathbb{R}^2\backslash\Bar{D}$, and the total field $u_{\rm tot}$ is given by the superposition of the incident field $\phi$ and the scattered field $u$, i.e.,
\begin{eqnarray}
    u_{\rm tot}(\mathbf{x},\bGamma): = \phi(\mathbf{x}) + u(\mathbf{x},\bGamma).
\end{eqnarray}
Here, we use $u(\mathbf{x},\bGamma)$ to emphasize the dependence of the scattered field on the boundary $\bGamma$, and abbreviate it as $u(\mathbf{x})$ when no ambiguity arises. The total field $u_{\rm tot}$ satisfies the Helmholtz equation
\begin{eqnarray}\label{scattering_eqn2}
    \nabla\cdot\alpha\nabla u_{\rm tot} + k^2 u_{\rm tot} = 0, \quad{\rm in}\quad D_{\rm ex},
\end{eqnarray}
with one of the following boundary conditions:
\begin{eqnarray}\label{scattering_bound2}
    \begin{aligned}
        &{\rm sound-soft:}&&u_{\rm tot}|_{\bGamma} = 0,\\
        &{\rm sound-hard:}&&\mathbf{n}\cdot\alpha\nabla u_{\rm tot}|_{\bGamma} = 0,\\
        &{\rm impedance:}&&\left(\mathbf{n}\cdot\alpha\nabla u_{\rm tot} + i\lambda u_{\rm tot}\right)|_{\bGamma} = 0,\\
    \end{aligned}
\end{eqnarray}
where  $\mathbf{n}$ is the unit outward normal to $\bGamma$, $\alpha$ is the medium parameter, and $\lambda$ is the impedance coefficient, both of which are assumed to be constant. In the case of a penetrable scatterer, the total field $u_{\rm tot}$ is defined in $D\cup D_{\rm ex}$ with $u_{\rm tot} = u$ in $D$. It satisfies the following transmission problem:
\begin{eqnarray}\label{scattering_eqn1}
\left\{
    \begin{aligned}
        &\nabla\cdot\alpha\nabla u_{\rm tot} + k^2 u_{\rm tot} = 0, \quad{\rm in}\quad D\cup D_{\rm ex},\\
        &\left[u_{\rm tot}\right]_{\bGamma} = 0,\\
         &\left[\mathbf{n}\cdot\alpha\nabla u_{\rm tot}\right]_{\bGamma} = 0,
    \end{aligned}
\right.
\end{eqnarray}
where $[\cdot]_{\bGamma}$ represents the jump across $\bGamma$. 

When $r = |\mathbf{x}|$ approaches infinity, in both impenetrable and penetrable cases, the scattered field $u$ satisfies the Sommerfeld radiation condition:
\begin{eqnarray}\label{rad_cond}
        \lim_{r\rightarrow\infty}\sqrt{r}\Big(\frac{\partial u}{\partial r} - iku\Big) = 0.
\end{eqnarray}

  Assume that the boundary of $D$ is perturbed by a collection of velocity fields $\mathbf{v}_{[m]}$, where $\mathbf{v}_{[m]}:=[\mathbf{v}_1,\mathbf{v}_2,...,\mathbf{v}_m]$ consists of $m$ velocity fields, with each $\mathbf{v}_j\in C^{\infty}(\mathbb{R}^2,\mathbb{R}^2)$. We assume each velocity field has a compact support, confined to a band-like region containing the boundary $\bGamma$~\cite{sokolowski1992introduction}. Let $\epsilon_{[m]}: = [\epsilon_1,\epsilon_2,...,\epsilon_m]$ be $m$ independent perturbation parameters.
  The perturbed boundary $\bGamma_{\rm per}$ is given by
\begin{eqnarray}\label{Gammav}
    \bGamma_{\rm per} = \bGamma + \sum_{j = 1}^{m}\epsilon_j\mathbf{v}_j.
\end{eqnarray}
The corresponding perturbed field $u(\mathbf{x},\bGamma_{\rm per})$ can be approximated by
\begin{eqnarray}\label{shapeTaylorMultiV}
    \begin{aligned}
        u(\mathbf{x},\bGamma_{\rm per}) \approx& u(\mathbf{x}) + \sum_{i = 1}^m\epsilon_i\delta_{\mathbf{v}_i}u(\mathbf{x}) + \frac{1}{2!}\sum_{i,j=1}^m\epsilon_i\epsilon_j\delta_{[\mathbf{v}_i,\mathbf{v}_j]}u(\mathbf{x})\\
        &+ ... + \frac{1}{N!}\sum_{i_1,...,i_N = 1}^m\left(\prod \limits_{j=1}^N\epsilon_{i_j}\right)\delta_{[\mathbf{v}_{i_1},...,\mathbf{v}_{i_N}]}u(\mathbf{x}),
    \end{aligned}
\end{eqnarray}
where $\delta_{[\mathbf{v}_{i_1},...,\mathbf{v}_{i_N}]}u$ for $N = 0,1,2,\dots$ is the $N$th order shape derivative with respect to (w.r.t.) velocity fields $\mathbf{v}_{i_1},...,\mathbf{v}_{i_N}$. Specifically, $u$ is the $0$th order shape derivative, and $\delta_{\mathbf{v}_j}u$ with $j = 1,2,\dots m$ are the first order shape derivatives. The right-hand side of equation~\eqref{shapeTaylorMultiV} is the shape Taylor expansion  of the scattered field $u$. For the special case of $m = 1$,  the expansion is given by equation~\eqref{shapetaylor}. Therefore, the computation of shape Taylor expansion is equivalent to computing the shape derivatives of each order. Readers are referred to~\cite{hiptmair2018shape} for the first order shape derivatives under the boundary conditions specified by equations~\eqref{scattering_bound2} and~\eqref{scattering_eqn1}.  Here we give the formulas for shape derivatives of arbitrary order. 

We begin by defining the normal Dirichlet trace and Neumann trace of a field $u$ as
\begin{eqnarray}
    \mathbf{Tr}^\mathcal{D}(u): = \mathbf{n}u|_{\bGamma},\qquad\mathbf{Tr}^\mathcal{N}(u): = \mathbf{n}(\mathbf{n}\cdot\alpha\nabla u)|_{\bGamma}.
\end{eqnarray}
We introduce two shape derivative operators, denoted as $\delta_{\mathbf{v}}^{u}$ and $\delta_{\mathbf{v}}^{\mathbf{n}}$, to represent the shape differential on the field $u$ and the normal vector $\mathbf{n}$, respectively. They are given in the form of
\begin{eqnarray}
    \delta_{\mathbf{v}}^{u}(\mathbf{n}u|_{\bGamma}) = \mathbf{n}\delta_{\mathbf{v}}u|_{\bGamma},\qquad \delta_{\mathbf{v}}^{\mathbf{n}}(\mathbf{n}u|_{\bGamma}) = \delta_{\mathbf{v}}\mathbf{n}u|_{\bGamma}.
\end{eqnarray}
Based on these notations, the following two theorems establish the recursive relations for the shape derivatives in acoustic scattering problems~\cite{bao2025shape}.

\begin{theorem}[Impenetrable cases]\label{thmImpenetrable}
    Let $u_{\rm tot}$ be the solution of~\eqref{scattering_eqn2} and $u$ be the corresponding scattered field. Suppose the $N$th order shape derivative of $u$ w.r.t. $\mathbf{v}_{[N]}$ satisfies
    \begin{eqnarray}\label{eq_impen}
        \nabla\cdot\alpha\nabla\delta_{\mathbf{v}_{[N]}} u + k^2\delta_{\mathbf{v}_{[N]}} u = 0, \quad{\rm in}\quad D_{\rm ex},
    \end{eqnarray}
    with boundary conditions on $\bGamma$ given by one of the following:
    \begin{subequations}\label{bound_impen}
        \begin{align}
            &{\rm Sound-soft:}&&\delta_{\mathbf{v}_{[N]}} u|_{\bGamma} = \mathbf{n}\cdot\mathbf{Tr}^{\mathcal{D}}(\delta_{\mathbf{v}_{[N]}}u),\label{bound_impen1}\\
            &{\rm Sound-hard:}&&\mathbf{n}(\mathbf{n}\cdot\alpha\nabla \delta_{\mathbf{v}_{[N]}} u)|_{\bGamma} = \mathbf{Tr}^{\mathcal{N}}(\delta_{\mathbf{v}_{[N]}}u),\label{bound_impen2}\\
            &{\rm Impedance:}&&\mathbf{n}(\mathbf{n}\cdot\alpha\nabla \delta_{\mathbf{v}_{[N]}} u + i\lambda \delta_{\mathbf{v}_{[N]}} u)|_{\bGamma}\label{bound_impen3} \\&&&=\mathbf{Tr}^{\mathcal{N}}(\delta_{\mathbf{v}_{[N]}}u,\mathbf{n}) + i\lambda\mathbf{Tr}^{\mathcal{D}}(\delta_{\mathbf{v}_{[N]}}u),\notag
        \end{align}
    \end{subequations}
    where $\delta_{\mathbf{v}_{[N]}}u = u$ for $N=0$. Then the $N+1$th order shape derivative $\delta_{\mathbf{v}_{[N+1]}}u$ also satisfies equation~\eqref{eq_impen} and radiation condition \eqref{rad_cond}. For the boundary condition on $\bGamma$, under the sound-soft case, it is
    \begin{eqnarray}\label{bound_diri}
    \begin{aligned}
        \delta_{\mathbf{v}_{[N+1]}} u|_{\bGamma} =& - \mathbf{v}_{N+1}\cdot\nabla\delta_{\mathbf{v}_{[N]}}u + \mathbf{v}_{N+1}\cdot\nabla\mathbf{n}\cdot\mathbf{Tr}^\mathcal{D}(\delta_{\mathbf{v}_{[N]}}u)\\
        &+\delta_{\mathbf{v}_{N+1}}^u\left( \mathbf{n}\cdot\mathbf{Tr}^\mathcal{D}(\delta_{\mathbf{v}_{[N]}}u)\right).
    \end{aligned}
    \end{eqnarray}
    Under the sound-hard case, it is
    \begin{eqnarray}\label{bound_neum}
        \begin{aligned}
            \mathbf{n}(\mathbf{n}\cdot\alpha\nabla \delta_{\mathbf{v}_{[N+1]}} u)|_{\bGamma} =&-\mathbf{v}_{N+1}\nabla\cdot\mathbf{n}(\mathbf{n}\cdot\alpha\nabla \delta_{\mathbf{v}_{[N]}} u) - \delta_{\mathbf{v}_{N+1}}^{\mathbf{n}}\left(\mathbf{n}(\mathbf{n}\cdot\alpha\nabla \delta_{\mathbf{v}_{[N]}} u)\right)\\
         & + \mathbf{v}_{N+1}\nabla\cdot\mathbf{Tr}^\mathcal{N}(\delta_{\mathbf{v}_{[N]}}u) + \delta_{\mathbf{v}_{N+1}}^{\mathbf{n}}\left(\mathbf{Tr}^\mathcal{N}(\delta_{\mathbf{v}_{[N]}}u)\right)\\
         & + \delta_{\mathbf{v}_{N+1}}^{u}\left(\mathbf{Tr}^\mathcal{N}(\delta_{\mathbf{v}_{[N]}}u)\right).
        \end{aligned}
    \end{eqnarray}
    Under the impedance case, it is
    \begin{eqnarray}\label{bound_impe}
        \begin{aligned}
            &\mathbf{n}(\mathbf{n}\cdot\alpha\nabla \delta_{\mathbf{v}_{[N+1]}} u + i\lambda \delta_{\mathbf{v}_{[N+1]}} u)|_{\bGamma} \\
            =& -\mathbf{v}_{N+1}\nabla\cdot\left(\mathbf{n}(\mathbf{n}\cdot\alpha\nabla \delta_{\mathbf{v}_{[N]}}u + i\lambda \delta_{\mathbf{v}_{[N]}} u)\right)\\
            &- \delta_{\mathbf{v}_{N+1}}^\mathbf{n}\left(\mathbf{n}(\mathbf{n}\cdot\alpha\nabla\delta_{\mathbf{v}_{[N]}}u+i\lambda\delta_{\mathbf{v}_{[N+1]}}u)\right)\\
            & + \mathbf{v}_{N+1}\nabla\cdot\left(\mathbf{Tr}^{\mathcal{N}}(\delta_{\mathbf{v}_{[N]}}u) + i\lambda\mathbf{Tr}^{\mathcal{D}}(\delta_{\mathbf{v}_{[N]}}u)\right)\\
            & + \delta_{\mathbf{v}_{N+1}}^{u}\left(\mathbf{Tr}^{\mathcal{N}}(\delta_{\mathbf{v}_{[N]}}u)+ i\lambda\mathbf{Tr}^{\mathcal{D}}(\delta_{\mathbf{v}_{[N]}}u)\right) + \delta_{\mathbf{v}_{N+1}}^\mathbf{n}\left(\mathbf{Tr}^{\mathcal{N}}(\delta_{\mathbf{v}_{[N]}}u)\right).
        \end{aligned}
    \end{eqnarray}
\end{theorem}

\begin{theorem}[Penetrable case]\label{thmPenetrable}
    Let $u$ be the scattered field with $u_{\rm tot}$ satisfying equation ~\eqref{scattering_eqn1}, and $\delta_{\mathbf{v}_{[N]}} u$ be the $N$th order shape derivative w.r.t. $\mathbf{v}_{[N]}$ for $N = 0,1,\dots$, where $\delta_{\mathbf{v}_{[N]}} u = u$ for $N = 0$. Suppose $\delta_{\mathbf{v}_{[N]}} u$ satisfies
\begin{eqnarray}\label{Traneq1}
\left\{
    \begin{aligned}
        &\nabla\cdot\alpha\nabla\delta_{\mathbf{v}_{[N]}}u + k^2\delta_{\mathbf{v}_{[N]}}u = 0, \quad{\rm in}\quad D\cup D_{\rm ex},\\
        &\left[\delta_{\mathbf{v}_{[N]}}u\right]_{\bGamma} = \left[\mathbf{n}\cdot\mathbf{Tr}^\mathcal{D}(\delta_{\mathbf{v}_{[N]}}u)\right]_{\bGamma},\\
        &\left[\mathbf{n}(\mathbf{n}\cdot\alpha\nabla \delta_{\mathbf{v}_{[N]}} u)\right]_{\bGamma} = \left[\mathbf{Tr}^\mathcal{N}(\delta_{\mathbf{v}_{[N]}}u)\right]_{\bGamma}.
    \end{aligned}
    \right.
\end{eqnarray}
Then the $N+1$th order shape derivative $\delta_{\mathbf{v}_{[N+1]}}u$ satisfies the Helmholtz equation with transmission condition on $\bGamma$ given by
\begin{eqnarray}\label{bound_pen}
    \begin{cases}
        &\left[\delta_{\mathbf{v}_{[N+1]}}u\right]_{\bGamma}= - \mathbf{v}_{N+1}\cdot\nabla\left[\delta_{\mathbf{v}_{[N]}}u\right]_{\bGamma}\\
        &+ \mathbf{v}_{N+1}\cdot\nabla\left[\mathbf{n}\cdot\mathbf{Tr}^\mathcal{D}(\delta_{\mathbf{v}_{[N]}}u)\right]_{\bGamma}+ \delta_{\mathbf{v}_{N+1}}^{u}\left[\mathbf{n}\cdot\mathbf{Tr}^\mathcal{D}(\delta_{\mathbf{v}_{[N]}}u)\right]_{\bGamma},\\
         &\left[\mathbf{n}(\mathbf{n}\cdot\alpha\nabla \delta_{\mathbf{v}_{[N+1]}} u)\right]_{\bGamma}=-\mathbf{v}_{N+1}\nabla\cdot\left[\mathbf{n}(\mathbf{n}\cdot\alpha\nabla \delta_{\mathbf{v}_{[N]}} u)\right]_{\bGamma}\\
         &-\left[\mathbf{n}(\mathbf{n}\cdot\alpha\nabla \delta_{\mathbf{v}_{[N]}} u)\right]_{\bGamma} + \mathbf{v}_{N+1}\nabla\cdot\left[\mathbf{Tr}^\mathcal{N}(\delta_{\mathbf{v}_{[N]}}u)\right]_{\bGamma}\\
         &+ \delta_{\mathbf{v}_{N+1}}^{u}\left[\mathbf{Tr}^\mathcal{N}(\delta_{\mathbf{v}_{[N]}}u)\right]_{\bGamma}+\delta_{\mathbf{v}_{N+1}}^{\mathbf{n}}\left[\mathbf{Tr}^\mathcal{N}(\delta_{\mathbf{v}_{[N]}}u)\right]_{\bGamma}.
    \end{cases}
\end{eqnarray}
%Here the jump across $\bGamma$ and the differential operator are commutative, i.e., $\mathbf{v}\cdot\nabla[.]_{\bGamma} = [\mathbf{v}\cdot\nabla(.)]_{\bGamma}$ and $\mathbf{v}\nabla\cdot[.]_{\bGamma} = [\mathbf{v}\nabla\cdot(.)]_{\bGamma}$. 
\end{theorem}

In both impenetrable and penetrable cases, when $N\ge 1$, it is worth noting that the shape derivative of the total field $u_{\rm tot}$ is the same as the shape derivative of the scattered field $u$, because the incident field $\phi$ is independent of the shape of the scatterer. We will repeatedly use this fact in the following sections.

\section{Computation of the shape Taylor expansion}\label{ShapeTaylor} 

This section illustrates how to compute the shape derivatives of arbitrary order that constitute the shape Taylor expansion. In general, there are two types of incident waves. One is the plane wave given in the form of:
\begin{subequations}\label{incwave1}
    \begin{align}
        &\phi_{pl}(\mathbf{x},\mathbf{z}) = \exp{(ik\mathbf{x}\cdot\mathbf{z})} &&{\rm on}\quad\bGamma,\quad {\rm Dirichlet\ data},\label{scatteringsf}\\
        &\frac{\partial \phi_{pl}}{\partial\mathbf{n}} = ik(\mathbf{z}\cdot\mathbf{n})\exp(ik\mathbf{x}\cdot\mathbf{z}) &&{\rm on}\quad\bGamma,\quad {\rm Neumann\ data},\label{scatteringsh}
    \end{align}
\end{subequations}
where $\mathbf{z}$ is the incident direction. Another one is the point source given in the form of: 
\begin{subequations}\label{incwave2}
    \begin{align}
    \phi_{pt}(\mathbf{x},\mathbf{x}_{\rm s}) &= \frac{i}{4}H^{(1)}_0(k|\mathbf{x} - \mathbf{x}_{\rm s}|)&&{\rm on}\quad\bGamma,\quad {\rm Dirichlet\ data},\\
    \frac{\partial \phi_{pt}}{\partial\mathbf{n}}(\mathbf{x},\mathbf{x}_{\rm s}) &= -\frac{ik(\mathbf{x} - \mathbf{x}_{\rm s})\cdot\mathbf{n}}{4|\mathbf{x} - \mathbf{x}_{\rm s}|}H^{(1)}_1(k|\mathbf{x} - \mathbf{x}_{\rm s}|) &&{\rm on}\quad\bGamma,\quad {\rm Neumann\ data},
    \end{align}
\end{subequations}
 where $\mathbf{x}_{\rm s}$ is the source point and $H^{(1)}_j(\cdot)$ is the first kind Hankel function of order $j$, with $j=0,1$.  
 
 Let $\bGamma$ be a smooth closed curve in $\mathbb{R}^2$ with a perimeter of $L$. Denote $\bgamma:[0,L]\rightarrow\bGamma, s\mapsto\bgamma(s)$  the boundary parametrization mapping the arc length parameter $s$ to $\bGamma$. Let $\partial_{\btau}$ and $\partial_\mathbf{n}$ be the tangential and normal differential operators on $\bGamma$, respectively. We use the boundary integral equation method to solve the scattering problems \eqref{scattering_eqn2} and \eqref{scattering_eqn1}. Let $\rho: \bGamma\rightarrow\mathbb{C}$ be a density function defined on $\bGamma$. Recall that the Green's function for the Helmholtz equation in two dimensions is given by $\mathbf{G}(\mathbf{x},\mathbf{y})=\phi_{pt}(\mathbf{x},\mathbf{y})$. The single and double layer potential operators acting on $\rho$ are defined as
\begin{subequations}\label{definition12}
\begin{align}
    \big(\mathbf{S}\rho\big)(\mathbf{x}) &:= \int_{\bGamma}\mathbf{G}(\mathbf{x},\bgamma(s))\rho(s)ds, \quad\mathbf{x}\in\mathbb{R}^2\backslash\bGamma,\\
    \big(\mathbf{D}\rho\big)(\mathbf{x}) &:= \int_{\bGamma}\frac{\partial\mathbf{G}(\mathbf{x},\bgamma(s))}{\partial\mathbf{n}(s)}\rho(s)ds,\quad\mathbf{x}\in\mathbb{R}^2\backslash\bGamma,\label{doublelayer}
\end{align}
\end{subequations}
Let $\mathbf{x} = \bgamma(t)$ and $\mathbf{y}=\bgamma(s)$ with $t,s\in[0,L]$ when $\mathbf{x}$ and $\mathbf{y}$ are both on $\bGamma$, in which case $\mathbf{G}(\mathbf{x},\mathbf{y})$ can be read as $\mathbf{G}(t,s)$ without ambiguity. The corresponding single and double layer boundary operators are given by
\begin{eqnarray}
    \left(\mathcal{S}\rho\right)(t):=\int_{\bGamma}\mathbf{G}(t,s)\rho(s)ds,\qquad
    \left(\mathcal{D}\rho\right)(t):= \int_{\bGamma}\frac{\partial\mathbf{G}(t,s)}{\partial\mathbf{n}(s)}\rho(s)ds.
\end{eqnarray}
Based on the potential theory~\cite{colton2019inverse}, when $\mathbf{x}\notin \bGamma$ approaches the boundary $\bGamma$, it holds:
\begin{eqnarray}\label{boundpotential}
\mathbf{S}\rho\rightarrow\mathcal{S}\rho,\quad\mathbf{D}\rho\rightarrow\pm\frac{1}{2}\rho+\mathcal{D}\rho,
\end{eqnarray}
where `$+$' and `$-$' correspond to $\mathbf{x}\in D_{\rm ex}$ and $\mathbf{x}\in D$, respectively. 

As established in~\cite{sokolowski1992introduction},the normal velocity fields are sufficient to describe the shape perturbation. Therefore, without loss of generality, we assume the boundary is perturbed by velocity fields of the form  $v(s)\mathbf{n}(s)$, where $v$ is a smooth function on $\bGamma$. In particular, the perturbed boundary $\bGamma_{\rm per}$ by a single velocity field is parameterized by 
\begin{eqnarray}\label{shapeper}
    \bgamma_{\rm per}(s) = \bgamma(s) + \epsilon v(s)\mathbf{n}(s),
\end{eqnarray}
 where $\epsilon$ is sufficiently small. 

In the remainder of this section, we will derive the explicit forms of shape Taylor expansions given by equation~\eqref{shapeTaylorMultiV} for the four types of boundary conditions. 

\subsection{Sound-soft boundary}\label{ssbound}

Recall the boundary condition on $\bGamma$ for the total field $u_{\rm tot}$ in sound-soft scattering problems is: 
\begin{eqnarray}
    u_{\rm tot} = \mathbf{n}\cdot\mathbf{Tr}^{\mathcal{D}}(u_{\rm tot}) = \phi + u = 0\qquad{\rm on}\quad\bGamma.
\end{eqnarray}
Setting $N=0$ in equation~\eqref{bound_diri}, the boundary condition for the first order shape derivative $\delta_\mathbf{v} u$ is: 
\begin{eqnarray}\label{vecBD1}
   \delta_{\mathbf{v}} u = -\mathbf{v}\cdot\nabla\mathbf{n}\cdot\mathbf{Tr}^{\mathcal{D}}(u_{\rm tot}) = -v\frac{\partial \phi}{\partial\mathbf{n}} -v\frac{\partial u}{\partial\mathbf{n}}  \qquad{\rm on}\quad\bGamma.
\end{eqnarray}
Let $\mathbf{w}:=w(s)\mathbf{n}(s)$ be another velocity field, by setting $N=1$ in equation~\eqref{bound_diri}, the boundary condition for the second order shape derivative $\delta_{[\mathbf{v},\mathbf{w}]}u$ w.r.t. $\mathbf{v}$ and $\mathbf{w}$ is
\begin{eqnarray}\label{vecBD2}
\begin{aligned}
    \delta_{[\mathbf{v},\mathbf{w}]}u =&-\mathbf{w}\cdot\nabla\delta_{\mathbf{v}}u+\mathbf{w}\cdot\nabla\left(-v\frac{\partial\phi}{\partial\mathbf{n}} -v\frac{\partial u}{\partial\mathbf{n}}\right)+\delta_{\mathbf{w}}^u\left(-v\frac{\partial\phi}{\partial\mathbf{n}} -v\frac{\partial u}{\partial\mathbf{n}}\right)\\
    &-vw\frac{\partial^2 \phi}{\partial\mathbf{n}^2} - v\frac{\partial\delta_{\mathbf{w}}u}{\partial\mathbf{n}} - w\frac{\partial\delta_{\mathbf{v}}u}{\partial\mathbf{n}} -vw\frac{\partial^2 u}{\partial\mathbf{n}^2} \qquad{\rm on}\quad\bGamma.
\end{aligned}
\end{eqnarray}
In general, for $N>2$, the shape derivative w.r.t. $N$ velocity fields $\mathbf{v}_{[N]}= [\mathbf{v}_1, \mathbf{v}_2,\dots,$ $\mathbf{v}_N]$ can be recursively derived through equation~\eqref{bound_diri}. If we assume
 that all velocity fields $\mathbf{v}_j$ for $j = 1,\dots,N$ in $\mathbf{v}_{[N]}$ are the same. The boundary condition for the $N$th order shape derivative $\delta_{\mathbf{v}_{[N]}} u$ is given by
\begin{eqnarray}\label{vecBDN1}
    \delta_{\mathbf{v}_{[N]}} u = -\left(v\frac{\partial}{\partial\mathbf{n}}\right)^N \phi - \sum_{j = 0}^{N-1}\binom{N}{j}\left(v\frac{\partial}{\partial\mathbf{n}}\right)^{N-j}\delta_{\mathbf{v}_{[j]}} u  \qquad{\rm on}\quad\bGamma.
\end{eqnarray}

\subsection{Sound-hard boundary}\label{shbound}
The boundary condition for $u_{\rm tot}$ in sound-hard scattering problems is
\begin{eqnarray}\label{bound311}
    \alpha \frac{\partial u_{\rm tot}}{\partial\mathbf{n}} =  \alpha\frac{\partial \phi}{\partial\mathbf{n}} + \alpha\frac{\partial u}{\partial\mathbf{n}} = 0 \qquad{\rm on}\quad\bGamma.
\end{eqnarray}
 Given $N = 0$ in equation~\eqref{bound_neum}, following Theorem \ref{thmImpenetrable}, the first order shape derivative w.r.t. $\mathbf{v}$ satisfies
\begin{eqnarray}\label{bound_neum_11}
\begin{aligned}
     \mathbf{n}\left(\mathbf{n}\cdot\alpha\nabla\delta_{\mathbf{v}}u\right) =& -\mathbf{v}\nabla\cdot\left(\mathbf{n}\alpha\frac{\partial u_{\rm tot}}{\partial\mathbf{n}}\right) - \delta_{\mathbf{v}}^\mathbf{n}\left(\mathbf{n}\alpha\frac{\partial u_{\rm tot}}{\partial\mathbf{n}}\right)\\
     =&-\mathbf{v}\alpha\frac{\partial^2 u_{\rm tot}}{\partial\mathbf{n}^2} - \mathbf{v}\bkappa\alpha\frac{\partial u_{\rm tot}}{\partial\mathbf{n}} - \mathbf{n}\alpha\frac{\partial u_{\rm tot}}{\partial\delta_{\mathbf{v}}\mathbf{n}},
\end{aligned}
\end{eqnarray}
where $\bkappa$ represents the curvature of $\bGamma$, which follows from $\nabla\cdot\mathbf{n} = \bkappa$. To see the second identity in equation~\eqref{bound_neum_11}, we note that
\begin{eqnarray}\label{eqn313}
    \delta_{\mathbf{v}}^{\mathbf{n}}\left(\mathbf{n}\alpha\frac{\partial u_{\rm tot}}{\partial\mathbf{n}}\right) = \delta_{\mathbf{v}}\mathbf{n}\alpha\frac{\partial u_{\rm tot}}{\partial\mathbf{n}} + \mathbf{n}\alpha\frac{\partial u_{\rm tot}}{\partial\delta_{\mathbf{v}}\mathbf{n}}.
\end{eqnarray}
The term $\delta_{\mathbf{v}}\mathbf{n}\alpha\frac{\partial u_{\rm tot}}{\partial\mathbf{n}}$ can be omitted due to the boundary condition~\eqref{bound311}. To obtain an explicit formula for the second term in the right side of equation \eqref{eqn313}, we recall the arc length parametrization of $\bGamma$, given by $\bgamma(s):=[\gamma_1(s),\gamma_2(s)]^\top$. This parametrization yields the unit tangential vector $\btau(s) = [\dot{\gamma_1}(s),\dot{\gamma_2}(s)]^\top$ and the unit normal vector $\mathbf{n}(s) = [\dot{\gamma}_2(s),-\dot{\gamma}_1(s)]^\top$. Since $\btau(s)$ and $\mathbf{n}(s)$ remain orthogonal for any $s$, determining $\delta_{\mathbf{v}}\mathbf{n}$ is equivalent to determining $\delta_{\mathbf{v}}\btau$. According to the perturbed boundary expressed in equation~\eqref{shapeper}, the parametrization of the perturbed unit tangential vector is given by
\begin{eqnarray}\label{tauper}
    \btau_{\rm per}(s) = \frac{\btau(s) + \epsilon v(s)\bkappa(s)\btau(s) + \epsilon\dot{v}(s)\mathbf{n}(s)}{\sqrt{\big(1 + \epsilon v(s)\bkappa(s)\big)^2 + \big(\epsilon\dot{v}(s)\big)^2}},
\end{eqnarray}
where we use the identity $\dot{\mathbf{n}} = \bkappa\btau$~\cite{delfour2011shapes}. From the first order Taylor expansion of equation~\eqref{tauper} w.r.t. $\epsilon$, we have
\begin{eqnarray}\label{deltavn}
    \delta_{\mathbf{v}}\btau = \lim_{\epsilon\rightarrow0}\frac{\btau_{\rm per} - \btau}{\epsilon} = \dot{v}\mathbf{n}\ \Rightarrow\ \delta_{\mathbf{v}}\mathbf{n} = -\dot{v}\btau.
\end{eqnarray}
Substituting equation~\eqref{deltavn} into~\eqref{bound_neum_11}, we obtain the Neumann boundary condition for the shape derivative $\delta_{\mathbf{v}}u$
\begin{eqnarray}\label{Neumanncon1}
    \alpha\frac{\partial\delta_{\mathbf{v}}u}{\partial\mathbf{n}} = -\left(\mathbf{v}\cdot\mathbf{n}\right)\alpha\left(\frac{\partial^2 \phi}{\partial\mathbf{n}^2} + \frac{\partial^2 u}{\partial\mathbf{n}^2}\right) + \dot{v}\alpha\left(\frac{\partial \phi}{\partial\btau} + \frac{\partial u}{\partial\btau}\right).
\end{eqnarray}

  By setting $N = 1$ in equation~\eqref{bound_neum} and using equation~\eqref{bound_neum_11}, it yields the second order shape derivative $\delta_{[\mathbf{v},\mathbf{w}]}u$ w.r.t. the velocity fields $\mathbf{v}$ and $\mathbf{w}$ as
\begin{eqnarray}\label{bound_Neum_22}
\begin{aligned}
    \mathbf{n}\alpha\frac{\partial\delta_{[\mathbf{v},\mathbf{w}]}u}{\partial\mathbf{n}} =&-\mathbf{w}\alpha\frac{\partial^2 \delta_{\mathbf{v}}u}{\partial\mathbf{n}^2} - \mathbf{w}\alpha\bkappa\frac{\partial \delta_{\mathbf{v}}u}{\partial\mathbf{n}} - \mathbf{n}\alpha\frac{\partial \delta_{\mathbf{v}}u}{\partial\delta_{\mathbf{w}}\mathbf{n}}\\
    &+\mathbf{w}\nabla\cdot\mathbf{Tr}^{\mathcal{N}}\left(\delta_{\mathbf{v}}u\right) +\delta_{\mathbf{w}}^{\mathbf{n}}\mathbf{Tr}^{\mathcal{N}}\left(\delta_{\mathbf{v}}u\right)+\delta_{\mathbf{w}}^{u}\mathbf{Tr}^{\mathcal{N}}\left(\delta_{\mathbf{v}}u\right)\\
    =&-\mathbf{w}\alpha\frac{\partial^2 \delta_{\mathbf{v}}u}{\partial\mathbf{n}^2} - \mathbf{w}\alpha\bkappa\frac{\partial \delta_{\mathbf{v}}u}{\partial\mathbf{n}} - \mathbf{n}\alpha\frac{\partial \delta_{\mathbf{v}}u}{\partial\delta_{\mathbf{w}}\mathbf{n}}\\
        &-\mathbf{w}v\alpha\frac{\partial^3 u_{\rm tot}}{\partial \mathbf{n}^3} - \mathbf{w}v\alpha\bkappa\frac{\partial^2 u_{\rm tot}}{\partial\mathbf{n}^2} - \mathbf{w}\alpha\frac{\partial^2u_{\rm tot}}{\partial\mathbf{n}\partial\delta_{\mathbf{v}}\mathbf{n}}\\
        &-\mathbf{w}v\bkappa\alpha\frac{\partial^2 u_{\rm tot}}{\partial\mathbf{n}^2} - \mathbf{w}v\bkappa^2\alpha\frac{\partial u_{\rm tot}}{\partial\mathbf{n}} - \mathbf{w}\bkappa\alpha\frac{\partial u_{\rm tot}}{\partial\delta_{\mathbf{v}}\mathbf{n}}\\
        &-\mathbf{v}\alpha\frac{\partial^2 \delta_{\mathbf{w}}u}{\partial\mathbf{n}^2} - \mathbf{v}\bkappa\alpha\frac{\partial \delta_{\mathbf{w}}u}{\partial\mathbf{n}} - \mathbf{n}\alpha\frac{\partial \delta_{\mathbf{w}}u}{\partial\delta_{\mathbf{v}}\mathbf{n}}\\
        &-2\mathbf{v}\alpha\frac{\partial^2 u_{\rm tot}}{\partial\mathbf{n}\partial\delta_{\mathbf{w}}\mathbf{n}} - \mathbf{v}\bkappa\alpha\frac{\partial u_{\rm tot}}{\partial\delta_{\mathbf{w}}\mathbf{n}} - \mathbf{n}\alpha\frac{\partial u_{\rm tot}}{\partial\delta_{\left[\mathbf{v},\mathbf{w}\right]}\mathbf{n}}.
\end{aligned}
\end{eqnarray}
The derivation of equation~\eqref{bound_Neum_22} is similar to that of equation~\eqref{bound_neum_11}, with the only difficulty being the second order derivative $\delta_{[\mathbf{v},\mathbf{w}]}\mathbf{n}$. By using the second order Taylor expansion of equation~\eqref{tauper} w.r.t. $\epsilon$, we obtain
\begin{eqnarray}\label{de2btau}
    \delta_{[\mathbf{v},\mathbf{v}]}\btau = 2\lim_{\epsilon\rightarrow0}\frac{\btau_{\rm per} - \btau - \epsilon\delta_{\mathbf{v}}\btau}{\epsilon^2} = ( 3v^2\bkappa^2 - \dot{v}^2 )\btau - v\dot{v}\bkappa\mathbf{n}.
\end{eqnarray}
Replacing $\epsilon\mathbf{v}$ with $\epsilon_1\mathbf{v} + \epsilon_2\mathbf{w}$ in equation~\eqref{tauper} and combining equation~\eqref{de2btau}, we obtain
\begin{eqnarray}
    \delta_{[\mathbf{v},\mathbf{w}]}\btau = \left( 3vw\bkappa^2 - \dot{v}\dot{w}\right)\btau - \frac{v\dot{w}+w\dot{v}}{2}\bkappa\mathbf{n}.
\end{eqnarray}
Thus, the second order shape derivative of the normal vector $\mathbf{n}$ w.r.t. $\mathbf{v}$ and $\mathbf{w}$ is given by
\begin{eqnarray}
    \delta_{[\mathbf{v},\mathbf{w}]}\mathbf{n} = (3vw\bkappa^2 - \dot{v}\dot{w})\mathbf{n} + \frac{v\dot{w}+w\dot{v}}{2}\bkappa\btau.
\end{eqnarray}
Therefore the terms involving $\delta_{[\mathbf{v},\mathbf{w}]}\mathbf{n}$ in equation~\eqref{bound_Neum_22} can be computed as
\begin{eqnarray}\label{de2norm}
    \mathbf{n}\alpha\frac{\partial u_{\rm tot}}{\partial\delta_{[\mathbf{v},\mathbf{w}]}\mathbf{n}} = \frac{\mathbf{v}\dot{w}+\mathbf{w}\dot{v}}{2}
    \bkappa\alpha\left(\frac{\partial \phi}{\partial\btau}+\frac{\partial u}{\partial\btau}\right).
\end{eqnarray}
Considering the boundary condition for $u_{\rm tot}$, we get
\begin{eqnarray}
    \mathbf{w}\alpha\frac{\partial^2u_{\rm tot}}{\partial\mathbf{n}\partial\delta_{\mathbf{v}}\mathbf{n}} = \mathbf{v}\alpha\frac{\partial^2 u_{\rm tot}}{\partial\mathbf{n}\partial\delta_{\mathbf{w}}\mathbf{n}}=\mathbf{w}v\bkappa^2\alpha\frac{\partial u_{\rm tot}}{\partial\mathbf{n}}=0.
\end{eqnarray}
Finally, by simplifying equation~\eqref{bound_Neum_22}, the boundary condition for $\delta_{[\mathbf{v},\mathbf{w}]}u$ can be given by
\begin{eqnarray}\label{NeumannCon2}
    \begin{aligned}
        \alpha\frac{\partial\delta_{[\mathbf{v},\mathbf{w}]}u}{\partial\mathbf{n}}=& -w\alpha\frac{\partial^2\delta_{\mathbf{v}}u}{\partial\mathbf{n}^2} - w\bkappa\alpha\frac{\partial \delta_{\mathbf{v}}u}{\partial\mathbf{n}}+\dot{w}\alpha\frac{\partial\delta_{\mathbf{v}}u}{\partial\btau}\\
        &-wv\alpha\left(\frac{\partial^3 \phi}{\partial\mathbf{n}^3} + \frac{\partial^3 u}{\partial\mathbf{n}^3}\right) - 2wv\bkappa\alpha\left(\frac{\partial^2 \phi}{\partial\mathbf{n}^2} + \frac{\partial^2 u}{\partial\mathbf{n}^2}\right)\\
        &+ w\dot{v}\bkappa\alpha\left(\frac{\partial \phi}{\partial\btau} + \frac{\partial u}{\partial\btau}\right) -v\alpha\frac{\partial^2\delta_{\mathbf{w}}u}{\partial\mathbf{n}^2} - v\bkappa\alpha\frac{\partial \delta_{\mathbf{w}}u}{\partial\mathbf{n}}+\dot{v}\alpha\frac{\partial\delta_{\mathbf{w}}u}{\partial\btau}\\
        &+ v\dot{w}\bkappa\alpha\left(\frac{\partial \phi}{\partial\btau} + \frac{\partial u}{\partial\btau}\right) - \frac{v\dot{w}+w\dot{v}}{2}
    \bkappa\alpha\left(\frac{\partial \phi}{\partial\btau}+\frac{\partial u}{\partial\btau}\right).
    \end{aligned}
\end{eqnarray}
The higher order shape derivative can be derived by following the same procedure.

\subsection{Impedance and transmission boundaries}\label{ImandTran}
According to the recurrence formulas~\eqref{bound_impe} and~\eqref{bound_pen}, computing the shape derivatives $\delta_{\mathbf{v}}u$ and $\delta_{[\mathbf{v},\mathbf{w}]}u$ for impedance and transmission scattering problems is essentially the same as in sound-soft and sound-hard problems. In particular, the total field in impedance scattering problems satisfies 
\begin{eqnarray}
    \mathbf{n}\left(\mathbf{n}\cdot\alpha\nabla u_{\rm tot} + i\lambda u_{\rm tot}\right) = 0\quad{\rm on}\quad\bGamma.
\end{eqnarray}
By setting $N=0$ in the recurrence formula~\eqref{bound_impe}, we obtain
\begin{eqnarray}\label{boundimp1}
\begin{aligned}
    &\mathbf{n}\left(\mathbf{n}\cdot\alpha\nabla\delta_{\mathbf{v}}u + i\lambda\delta_{\mathbf{v}}u\right)\\
    =&-\mathbf{v}\nabla\cdot\left(\mathbf{n}\left(\mathbf{n}\cdot\alpha\nabla u_{\rm tot}+i\lambda u_{\rm tot}\right)\right) - \delta_{\mathbf{v}}^{\mathbf{n}}\left(\mathbf{n}\left(\mathbf{n}\cdot\alpha\nabla u_{\rm tot} + i\lambda u_{\rm tot}\right)\right)\\
    =&-\mathbf{v}\left(\alpha\frac{\partial^2 u_{\rm tot}}{\partial\mathbf{n}^2} + i\lambda\frac{\partial u_{\rm tot}}{\partial\mathbf{n}}\right) - \mathbf{v}\bkappa\left(\alpha\frac{\partial u_{\rm tot}}{\partial\mathbf{n}} + i\lambda u_{\rm tot}\right)-\mathbf{n}\alpha\frac{\partial u_{\rm tot}}{\partial\delta_{\mathbf{v}}\mathbf{n}}\\
    =&\mathbf{Tr}^{\mathcal{N}}(\delta_{\mathbf{v}}u) + i\lambda\mathbf{Tr}^{\mathcal{D}}(\delta_{\mathbf{v}}u)\quad{\rm on}\quad\bGamma.
\end{aligned}
\end{eqnarray}
Thus, the boundary condition for $\delta_{\mathbf{v}}u$ is given by
\begin{eqnarray}
    \begin{aligned}
        &\alpha\frac{\partial \delta_{\mathbf{v}}u}{\partial\mathbf{n}} + i\lambda\delta_{\mathbf{v}} u \\
        =& -v\alpha\left(\frac{\partial^2 \phi}{\partial\mathbf{n}^2} + \alpha\frac{\partial^2 u}{\partial\mathbf{n}^2}\right) - i\lambda v\left(\frac{\partial \phi}{\partial\mathbf{n}} + \frac{\partial u}{\partial\mathbf{n}}\right)+\dot{v}\alpha\left(\frac{\partial \phi}{\partial\btau} + \frac{\partial \phi}{\partial\btau}\right)
    \end{aligned}
\end{eqnarray}
For $N = 1$, we introduce another velocity field $\mathbf{w}$. It holds
\begin{eqnarray}\label{bound_impe_2}
    \begin{aligned}
        &\mathbf{n}\left(\mathbf{n}\cdot\alpha\nabla\delta_{[\mathbf{v},\mathbf{w}]}u + i\lambda\delta_{[\mathbf{v},\mathbf{w}]}u\right)\\
        =&-\mathbf{w}\left(\alpha\frac{\partial^2 \delta_{\mathbf{v}}u}{\partial\mathbf{n}^2} + i\lambda\frac{\partial \delta_{\mathbf{v}}u}{\partial\mathbf{n}}\right) - \mathbf{w}\bkappa\left(\alpha\frac{\partial \delta_{\mathbf{v}}u}{\partial\mathbf{n}} + i\lambda \delta_{\mathbf{v}}u\right)-\mathbf{n}\alpha\frac{\partial \delta_{\mathbf{v}}u}{\partial\delta_{\mathbf{w}}\mathbf{n}}\\
        &-\mathbf{w}\nabla\cdot\left(\mathbf{Tr}^{\mathcal{N}}(\delta_{\mathbf{v}}u) + i\lambda\mathbf{Tr}^{\mathcal{D}}(\delta_{\mathbf{v}}u)\right)\\
        &-\delta_{\mathbf{w}}^{\mathbf{n}}\left(\mathbf{Tr}^{\mathcal{N}}(\delta_{\mathbf{v}}u) + i\lambda\mathbf{Tr}^{\mathcal{D}}(\delta_{\mathbf{v}}u)\right)\\
        &-\delta_{\mathbf{w}}^{u}\left(\mathbf{Tr}^{\mathcal{N}}(\delta_{\mathbf{v}}u) + i\lambda\mathbf{Tr}^{\mathcal{D}}(\delta_{\mathbf{v}}u)\right),\quad{\rm on}\quad\bGamma,
    \end{aligned}
\end{eqnarray}
where $\mathbf{Tr}^{\mathcal{N}}(\delta_{\mathbf{v}}u) + i\lambda\mathbf{Tr}^{\mathcal{D}}(\delta_{\mathbf{v}}u)$ is given by equation~\eqref{boundimp1}. Then the boundary condition for $\delta_{[\mathbf{v},\mathbf{w}]}u$ under the impedance case can be derived from equation~\eqref{bound_impe_2}, following the same steps as for boundary condition~\eqref{NeumannCon2} from equation~\eqref{bound_Neum_22}.

For penetrable scattering problems with transmission boundary conditions, considering the case of $N = 0$ in equation~\eqref{bound_pen}, we get
\begin{eqnarray}\label{bound_tran1}
    \left[\delta_{\mathbf{v}}u\right]_{\bGamma} = -\mathbf{v}\cdot\nabla\left[u_{\rm tot}\right]_{\bGamma} = -\left[\frac{\partial \phi}{\partial\mathbf{n}}+\frac{\partial u}{\partial\mathbf{n}}\right]_{\bGamma},
\end{eqnarray}
and
\begin{eqnarray}\label{bound_tran2}
    \begin{aligned}
    \left[\mathbf{n}\left(\mathbf{n}\cdot\alpha\nabla\delta_{\mathbf{v}}u\right)\right]_{\bGamma} =& -\mathbf{v}\nabla\cdot\left[\mathbf{n}\left(\mathbf{n}\cdot\alpha\nabla u_{\rm tot}\right)\right]_{\bGamma}-\delta_{\mathbf{v}}^{\mathbf{n}}\left[\mathbf{n}\left(\mathbf{n}\cdot\alpha\nabla u_{\rm tot}\right)\right]_{\bGamma}\\
    =&\left[-\mathbf{v}\alpha\left(\frac{\partial^2 \phi}{\partial\mathbf{n}^2}+\frac{\partial^2 u}{\partial\mathbf{n}^2}\right) + \mathbf{n}\dot{v}\alpha\left(\frac{\partial \phi}{\partial\btau}+\frac{\partial u}{\partial\btau}\right)\right]_{\bGamma}.
    \end{aligned}
\end{eqnarray}
 Comparing equation~\eqref{bound_tran1} with~\eqref{vecBD1}, and~\eqref{bound_tran2} with~\eqref{Neumanncon1}, one can see that the boundary conditions for $\delta_{\mathbf{v}}u$ are obtained directly from the conclusions of sound-soft and sound-hard cases, namely, by taking the difference of boundary data across the interface. The higher order shape derivatives for the penetrable problems follow the same principle as the first order case.

\subsection{Evaluating normal derivatives}\label{normalderi}
The primary components of the boundary conditions on $\bGamma$ for the shape derivatives given by equations~\eqref{vecBD2} and~\eqref{NeumannCon2} are the $N$th order normal and tangential derivatives of the scattered field $u$. While the tangential derivatives can be computed directly, the normal derivatives require a more involved approach. In this part, we derive the normal derivatives of arbitrary order of a field $u$ from the combinations of Dirichlet data $u|_{\bGamma}$ and Neumann data $\partial_{\mathbf{n}}u|_{\bGamma}$. Specifically, the proof of the following theorem provides a recursive method to compute these combinations~\cite{schulz1998computation,schwab1999extraction}.
\begin{theorem}\label{T31}
    Let $\bgamma\in C^{\infty}([0, S],\mathbb{R}^2)$ be the parameterization mapping from the arc length $s$ to the boundary $\bGamma$ and $u$ be the scattered field of problem~\eqref{scattering_eqn2}. Then the $N$th order normal derivative $\partial^N_{\mathbf{n}}u$ on $\bGamma$ is a linear combination of $\partial_{\btau}^j u$ and $\partial_{\btau}^j \partial_{\mathbf{n}}u$ with $0\le j\le N$.
\end{theorem}
\begin{remark}
     The coefficients of the combination only depend on the geometrical parameters of $\bGamma$ and the wavenumber $k$.
\end{remark}
\begin{proof}
    Consider the impenetrable case only and let $\mathbf{y}\in D_{\rm ex}$ be a point near $\bGamma$. We adopt the local coordinate representation similar to that used in~\cite{schulz1998computation}. Denote $(s,\eta)$ the coordinates of $\mathbf{y}$ in this curvilinear coordinate system, where $s$ is the arc length along $\bGamma$ and $\eta$ is the distance between $\mathbf{y}$ and $\bgamma(s)$. Then $\mathbf{y}$ is expressed as $\mathbf{y} = \bgamma(s) + \eta\mathbf{n}(s)$. Following~\cite{delfour2011shapes}, we have
   \begin{eqnarray}
       \frac{\partial \mathbf{y}}{\partial s} = \partial_s\bgamma + \eta\partial_s\mathbf{n} = (1+\eta\bkappa)\partial_s\bgamma: = \bchi\partial_s\bgamma,\quad\frac{\partial \mathbf{y}}{\partial\eta} = \mathbf{n}.
   \end{eqnarray}
   Without loss of generality, we set $\alpha\equiv1$ in equation~\eqref{scattering_eqn2}, which yields
   \begin{eqnarray}\label{N2}
       \begin{aligned}
           0 = \Delta u(\mathbf{y}) + k^2u(\mathbf{y}) = \frac{1}{\bchi^2}\partial^2_su - \frac{\eta\partial_s\bkappa}{\bchi^3}\partial_su + \partial_\eta^2u + \frac{\bkappa}{\bchi}\partial_\eta u + k^2u.
       \end{aligned}
   \end{eqnarray}
   Let $\eta\rightarrow0^+$ (i.e., $\mathbf{y}\rightarrow\bGamma^+$), we obtain
   \begin{eqnarray}
       \bchi\rightarrow1,\quad \partial^j_\eta u\rightarrow\partial^j_\mathbf{n}u,\quad \partial_s^ju\rightarrow \partial_{\btau}^ju,\quad\partial^j_s\partial_\eta u\rightarrow\partial^j_{\btau}\partial_{\mathbf{n}}u,\quad j\ge1.
   \end{eqnarray} 
    Thus the value of $\partial^2_\mathbf{n}u$ is given by
    \begin{eqnarray}\label{norm2}
        \partial^2_{\mathbf{n}} u = -\bkappa \partial_{\mathbf{n}}u - \partial^2_{\btau}u - k^2u\qquad{\rm on}\quad\bGamma.
    \end{eqnarray}    
    
    To derive $\partial_\mathbf{n}^Nu$ for $N \ge 3$, we take the derivative $\partial_\eta$ on both sides of equation~\eqref{N2} and substitute $\partial^2_\eta u$ using equation~\eqref{N2}. Suppose the linear combination for the $N$th order normal derivative $\partial^N_\eta u$ at $\mathbf{y}$ is   
    \begin{eqnarray}\label{lincomb}
       \partial^N_\eta u = \sum_{i = 0}^N a^N_i\partial^i_su + \sum_{j = 0}^{N-1} b^N_j\partial^j_s\partial_\eta u,
   \end{eqnarray}
   where, according to equation~\eqref{norm2}, for $N = 2$ it holds
   \begin{eqnarray}
          a^2_0 = -k^2,\quad a^2_1 = \frac{\eta\partial_s\bkappa}{\bchi^3},\quad a^2_2 = -\frac{1}{\bchi^2},\quad b^2_0 = -\frac{\bkappa}{\bchi},\quad b^2_1 = 0.
   \end{eqnarray}
    Then the $(N+1)$th order normal derivative of $u$ w.r.t. $\eta$ is given by
   \begin{eqnarray}\label{u_eta}
   \begin{aligned}
       \partial^{N+1}_\eta u =& \sum_{i = 0}^N a^N_i\partial^i_s\partial_\eta u + \sum_{j = 0}^{N-1} b^N_j\partial^j_s\partial_\eta^2u\\
       =&\sum_{i = 0}^N a^N_i\partial^i_s\partial_\eta u + \sum_{j = 0}^{N-1} b^N_j\partial^j_s\left(-k^2u + \frac{\eta\partial_s\bkappa}{\bchi^3}\partial_su - \frac{1}{\bchi^2}\partial^2_su - \frac{\bkappa}{\bchi}\partial_\eta u\right).
   \end{aligned}
   \end{eqnarray}
    Taking $\eta\rightarrow0^+$ in equation~\eqref{u_eta} yields
    \begin{eqnarray}
        \partial^{N+1}_\eta u\rightarrow\partial^{N+1}_\mathbf{n}u,\qquad\partial^i_su\rightarrow\partial^i_{\btau}u,\qquad\partial^j_s\partial_\eta u\rightarrow\partial^j_{\btau}\partial_{\mathbf{n}}u.
    \end{eqnarray}
\end{proof}

For numerical purposes, here we give the explicit formula for $\partial_{\mathbf{n}}^3 u$.
\begin{corollary}\label{Coro31}
    The third order normal derivative of the scattered field $u$ on the boundary $\bGamma$ is given by
    \begin{eqnarray}\label{coro32}
        \partial_\mathbf{n}^3u = 3\bkappa\partial^2_{\btau}u + \partial_{\btau}\bkappa\partial_{\btau}u + k^2\bkappa u - \partial^2_{\btau}\partial_{\mathbf{n}}u + (2\bkappa^2 - k^2)\partial_{\mathbf{n}}u.
    \end{eqnarray}  
\end{corollary}

When $\bGamma$ is a circle, in which $\partial_s\bkappa = \partial_s\bchi \equiv 0$,  the linear coefficients $a^N_j$ and $b^N_j$ in equation~\eqref{lincomb} can be greatly simplified. 
\begin{corollary}
    Let $\bGamma$ be a circle. Then the derivative of the scattered field $u$ near $\bGamma$ satisfies
\begin{eqnarray}\label{pspn}
    \partial^M_s\partial^N_\eta u = \sum_{i = 0}^Na^N_i\partial^{i + M}_su + \sum_{j = 0}^{N-1}b^N_j\partial^{j+M}_s\partial_\eta u.
\end{eqnarray}
with $\left[a^2_0, a^2_1, a^2_2,b^2_0,b^2_1\right] = \left[-k^2,0,-\frac{1}{\bchi^2},-\frac{\bkappa}{\bchi},0\right]$ for $N=2$. The recurrence coefficients from $N$ to $N+1$ for $N\ge2$ are given by
\begin{eqnarray}\label{CircleCoef}
\begin{aligned}
    &a^{N+1}_0=&&\partial_\eta a^N_0 + b^N_0a^2_0,\\
    &a^{N+1}_1=&&\partial_\eta a^N_1 + b^N_1a^2_0 + b^N_0a^2_1,\\
    &a^{N+1}_j=&&\partial_\eta a^N_j + b^N_ja^2_0 + b^N_{j-1}a^2_1 + b^N_{j-2}a^2_2,\quad j = 2,3,\dots,N-1,\\
    &a^{N+1}_N=&&\partial_\eta a^N_N + b^N_{N-1}a^2_1 + b^N_{N-2}a^2_2,\\
    &a^{N+1}_{N+1}=&&b^N_{N-1}a^2_2,
\end{aligned}
\end{eqnarray}
and
\begin{eqnarray}
    \begin{aligned}
        &b^{N+1}_0 =&& a^N_0 + \partial_\eta b^N_0 + b^N_0b^2_0,\\
        &b^{N+1}_j =&& a^N_j + \partial_\eta b^N_j + b^N_{j-1}b^2_1,\quad j = 1,2,\dots,N-1,\\
        &b^{N+1}_N =&& a^N_N + b^N_{N-1}b^2_1.\\
    \end{aligned}
\end{eqnarray}
\end{corollary}
The proof follows by a direct verification.

From the discussion above, we see that computing the shape derivatives requires both the Dirichlet and Neumann data of the scattered field on the boundary of the scatterer. Since these two types of data are not given simultaneously, we derive the relationship between $\partial^j_{\btau}u$ and $\partial^j_{\btau}\partial_{\mathbf{n}}u$ based on Green's theorem~\cite{colton2019inverse}.
\begin{theorem}\label{T32}
    Let $u$ be the scattered field of the scattering problem~\eqref{scattering_eqn2}. The tangential derivatives $\partial^j_{\btau}\partial_{\mathbf{n}}u$ and $\partial^j_{\btau}u$ on the boundary $\bGamma$ determine each other for any $j = 0,1,2,\dots$.
\end{theorem}
\begin{proof}
    For the case of $j = 0$, according to Green's representation theorem and the jump relation of the double layer potential operator given by equation~\eqref{boundpotential}, $u$ and $\partial_{\mathbf{n}}u$ satisfy
    \begin{eqnarray}\label{Green0}
        \int_{\bGamma}\mathbf{G}(t,s)\partial_{\mathbf{n}}u(s)ds = -\frac{1}{2}u(t) + \int_{\bGamma}\frac{\partial\mathbf{G}(t,s)}{\partial\mathbf{n}(s)}u(s)ds, \mbox{ on }\bGamma,
    \end{eqnarray}
    which can be represented as $\mathcal{S}\partial_{\mathbf{n}}u = -\frac{1}{2}u + \mathcal{D}u$.
    
    For $j = 1$, by taking derivatives w.r.t. $t$ on both sides of equation~\eqref{Green0} and integrating by parts, we get
    \begin{eqnarray}\label{Green1}
    \begin{aligned}
        &\int_{\bGamma}\partial_t\mathbf{G}(t,s)\partial_{\mathbf{n}}u(s)ds\\
        =& \int_{\bGamma}(\partial_t + \partial_s)\mathbf{G}(t,s)\partial_{\mathbf{n}}u(s)ds + \int_{\bGamma} \mathbf{G}(t,s)\partial_s\partial_{\mathbf{n}}u(s)ds\\
        =&-\frac{1}{2}\partial_tu(t) + \int_{\bGamma}(\partial_t+\partial_s)\frac{\partial\mathbf{G}(t,s)}{\partial\mathbf{n}(s)}u(s)ds + \int_{\bGamma}\frac{\partial\mathbf{G}(t,s)}{\partial\mathbf{n}(s)}\partial_su(s)ds.
    \end{aligned} 
    \end{eqnarray}
    Therefore, $\partial_{\btau}u$ and $\partial_{\btau}\partial_{\mathbf{n}}u$ satisfy
    \begin{eqnarray}
        \mathcal{S}\partial_{\btau}\partial_{\mathbf{n}}u + \mathcal{S}^{(1)}\partial_{\mathbf{n}}u = -\frac{1}{2}\partial_{\btau}u + \mathcal{D}\partial_{\btau}u + \mathcal{D}^{(1)}u.
    \end{eqnarray}
    Here the boundary operators $\mathcal{S}^{(j)}$, $\mathcal{D}^{(j)}$ for $j=1,2,\dots$ are defined by
    \begin{subequations}\label{movidiff}
    \begin{align}
         \left(\mathcal{S}^{(j)}\rho\right)(t): =& \int_{\bGamma}(\partial_t + \partial_s)^j\mathbf{G}(t,s)\rho(s)ds,\label{Single}\\
        \left(\mathcal{D}^{(j)}\rho\right)(t): =& \int_{\bGamma}(\partial_t + \partial_s)^j\frac{\partial\mathbf{G}(t,s)}{\partial\mathbf{n}(s)}\rho(s)ds.\label{Double}
    \end{align}
    \end{subequations}
    By induction, for $j>1$, $\partial^j_{\btau}u$ and $\partial^j_{\btau}\partial_{\mathbf{n}}u$ satisfy
    \begin{eqnarray}\label{un}
        \sum_{p=0}^{j}\mathcal{S}^{(p)}\partial^{j-p}_{\btau}\partial_{\mathbf{n}}u = -\frac{1}{2}\partial^{j}_{\btau}u + \sum_{l = 0}^{j}\mathcal{D}^{(l)}\partial^{j-l}_{\btau}u.
    \end{eqnarray}
\end{proof}

In the case when $\bGamma$ is a circle, the boundary operators $\mathcal{S}^{(j)}$ and $\mathcal{D}^{(j)}$ for $j\ge1$ are all zero. Thus equation~\eqref{un} can be simplified to
\begin{eqnarray}\label{CircleDtN}
    \mathcal{S}\partial^j_{\btau}\partial_{\mathbf{n}}u = \left(-\frac{1}{2}\mathcal{I} + \mathcal{D}\right)\partial^j_{\btau}u,\qquad j\in\mathbb{N}.
\end{eqnarray}

\begin{remark}
 All the results derived in this section can be extended to the three dimensional case. One simply needs to replace the curve-based geometry with the corresponding surface-based geometry.
\end{remark}

\section{Random boundary scattering problem}\label{Randomproblem}

This section exhibits an application of the shape Taylor expansion in uncertainty quantification. We formulate the scattering problem with random boundaries based on model~\eqref{scattering_eqn2}. The random perturbations on boundary $\bGamma$ are given by a set of independent random variables $\bomega:=[\omega_1,...,\omega_m]$. We assume that $\bomega$ belongs to an $m$ dimensional probability space $(\bOmega_P,\mathcal{F}, P)$, where the $j$th component $\omega_j$ follows a uniform distribution on $[-1,1]$. The random perturbed boundary is given by
\begin{eqnarray}\label{randbound}
    \bGamma_{\epsilon\bomega}: = \bGamma + \epsilon\sum_{j = 1}^m\omega_j\mathbf{v}_j.
\end{eqnarray}
The statistical properties of a random field can be described by the moments given by the following definition.
\begin{definition}\label{def1}
 Let $u$ be the scattered field for the scattering problem~\eqref{scattering_eqn2} with random boundary perturbations. The $n$th order moment for $u$ is defined as
    \begin{eqnarray}\label{mome}
    \mathbb{M}^n[u](\mathbf{x}) := \int_{\bOmega_P} \big[u(\mathbf{x},\bGamma_{\epsilon\bomega})\big]^nP(\bomega)d\bomega,\quad n\ge 1.
\end{eqnarray}
Specifically, the first order moment is the expectation
\begin{eqnarray}
    \mathbb{E}[u](\mathbf{x}): = \mathbb{M}^1[u](\mathbf{x}).
\end{eqnarray}
The $n$th order central moment is defined as
\begin{eqnarray}\label{momecen}
    \mathbb{M}^n_0[u](\mathbf{x}) := \int_{\bOmega_P} \big(u(\mathbf{x},\bGamma_{\epsilon\bomega}) - \mathbb{E}[u(\mathbf{x})]\big)^nP(\bomega)d\bomega,\quad n\ge 1.
\end{eqnarray}
Specifically, the second order central moment is the variance
\begin{eqnarray}
    \mathbb{VAR}[u](\mathbf{x}): = \mathbb{M}^2_0[u](\mathbf{x}).
\end{eqnarray}
\end{definition}

According to equation~\eqref{shapeTaylorMultiV}, the multivariate Taylor expansion under the random perturbed boundary $\bGamma_{\epsilon\bomega}$ is 
\begin{eqnarray}\label{multiTaylor}
    \begin{aligned}
        u(\mathbf{x},\bGamma_{\epsilon\bomega}) =& u(\mathbf{x}) + \epsilon\sum_{i = 1}^m\omega_i\delta_{\mathbf{v}_i}u(\mathbf{x}) + \frac{\epsilon^2}{2!}\sum_{i,j=1}^m\omega_i\omega_j\delta_{[\mathbf{v}_i,\mathbf{v}_j]}u(\mathbf{x})\\
        &+ ... + \frac{\epsilon^N}{N!}\sum_{i_1,...,i_N = 1}^m\left(\prod \limits_{j=1}^N\omega_{i_j}\right)\delta_{[\mathbf{v}_{i_1},...,\mathbf{v}_{i_N}]}u(\mathbf{x}) + \mathcal{O}\left(\epsilon^{N+1}\right).
    \end{aligned}
\end{eqnarray}
Since the variables $\omega_j$ are mutually independent and uniformly distributed over the interval $[-1,1]$, it follows that
\begin{eqnarray}\label{prop}
    \int_{\bOmega_P}\prod\limits_{n = 1}^q\omega_{i_n}P(\bomega)d\bomega = 0,
\end{eqnarray}
when $q$ is any positive odd integer. Combining equations~\eqref{randbound},~\eqref{multiTaylor} and~\eqref{prop}, we can estimate the moments of the random field based on the second order shape Taylor expansion.
\begin{theorem}
     Let $u(\mathbf{x},\bGamma_{\epsilon\bomega})$ be the scattered field of the scattering problem~\eqref{scattering_eqn2} with random boundaries. Then the first order moment satisfies
    \begin{eqnarray}
    \mathbb{E}[u](\mathbf{x}) = \mathbb{M}^1[u](\mathbf{x}) = u(\mathbf{x}) + \frac{\epsilon^2}{3!}\sum_{i = 1}^m\delta_{[\mathbf{v}_i,\mathbf{v}_i]}u(\mathbf{x}) + \mathcal{O}\left(\epsilon^4\right).
    \end{eqnarray}
    The second order moment satisfies
    \begin{eqnarray}
    \mathbb{M}^2[u](\mathbf{x}) = u^2(\mathbf{x}) + \frac{\epsilon^2}{3}\sum_{i = 1}^m\Big(\delta_{\mathbf{v}_i}u^2(\mathbf{x}) + u(\mathbf{x})\delta_{[\mathbf{v}_i,\mathbf{v}_i]}u(\mathbf{x})\Big) + \mathcal{O}\left(\epsilon^4\right).
    \end{eqnarray}
       Inductively, the $n$th order  moment with $n>2$ of $u$ satisfies
    \begin{eqnarray}\label{Es4Mn}
    \begin{aligned}
        \mathbb{M}^n[u](\mathbf{x}) =& u^n(\mathbf{x})+ \frac{\epsilon^2}{3}\sum_{i = 1}^m\left(\binom{n}{2}u^{n-2}(\mathbf{x})\delta_{\mathbf{v}_i}u^2(\mathbf{x}) + \binom{n}{1}u^{n-1}(\mathbf{x})\frac{1}{2}\delta_{[\mathbf{v}_i,\mathbf{v}_i]}u(\mathbf{x})\right)\\
        &+ \mathcal{O}\left(\epsilon^4\right).
    \end{aligned}
\end{eqnarray}    
\end{theorem}

In general, when estimating the $n$th moment of the random field using the $N$th order shape Taylor expansion, the approximation error is of order $\mathcal{O}(\epsilon^{N+2})$, where $N$ is a positive even integer.

 We also give the result for estimating the central moments based on the second order shape Taylor expansion.
\begin{theorem}\label{centralmom}
    Let $u(\mathbf{x},\bGamma_{\epsilon\bomega})$ be the total field of the scattering problems with random boundaries.
    The first order central moment is zero since that
\begin{eqnarray}\label{centmom_1}
    \mathbb{M}_0^1[u](\mathbf{x}) = \mathbb{E}[u(\mathbf{x},\bGamma_{\epsilon\bomega}) - \mathbb{E}[u](\mathbf{x})] \equiv 0.
\end{eqnarray}
The second order central moment, i.e., the variance, satisfies
\begin{eqnarray}\label{centmom_2}
    \begin{aligned}
        \mathbb{M}_0^2[u](\mathbf{x})
        =& \mathbb{E}\left[\left(u(\mathbf{x},\bGamma_{\epsilon\bomega}) - \mathbb{E}[u](\mathbf{x})\right)^2\right]\\
        =&\mathbb{E}\left[\left(\epsilon\sum_{i = 1}^m\omega_i\delta_{\mathbf{v}_i}u(\mathbf{x}) + \frac{\epsilon^2}{2!}\sum_{i, j = 1}^{m}\omega_i\omega_j\delta_{[\mathbf{v}_i,\mathbf{v}_j]}u(\mathbf{x})\right.\right.\\
        &- \left.\left.\frac{\epsilon^2}{3!}\sum_{i = 1}^m\delta_{[\mathbf{v}_i,\mathbf{v}_i]}u(\mathbf{x}) + \mathcal{O}\left(\epsilon^4\right)\right)^2\right]\\
        =&\frac{\epsilon^2}{3}\sum_{i = 1}^m\delta_{\mathbf{v}_i}u^2(\mathbf{x}) + \mathcal{O}\left(\epsilon^4\right).
    \end{aligned}
\end{eqnarray}
 For the higher order central moment, if $n$ is even, the $n$th order central moment satisfies
\begin{eqnarray}
    \begin{aligned}
        \mathbb{M}_0^n[u](\mathbf{x}) =&\mathbb{E}\left[\left(u(\mathbf{x},\bGamma_{\epsilon\bomega}) - \mathbb{E}[u](\mathbf{x})\right)^n\right]\\
        =&\mathbb{E}\left[\left(\epsilon\sum_{i = 1}^m\omega_i\delta_{\mathbf{v}_i}u(\mathbf{x}) + \frac{\epsilon^2}{2!}\sum_{i, j = 1}^{m}\omega_i\omega_j\delta_{[\mathbf{v}_i,\mathbf{v}_j]}u(\mathbf{x})\right.\right.\\
        &\left.\left.- \frac{\epsilon^2}{3!}\sum_{i = 1}^m\delta_{[\mathbf{v}_i,\mathbf{v}_i]}u(\mathbf{x}) + \mathcal{O}\left(\epsilon^4\right)\right)^n\right]\\
        =&\epsilon^n\mathbb{E}\left[\left(\sum_{i = 1}^m\omega_i\delta_{\mathbf{v}_i}u\right)^n\right](\mathbf{x}) + \mathcal{O}\left(\epsilon^{n+2}\right).
    \end{aligned}
\end{eqnarray}
If $n$ is odd, the central moment satisfies
\begin{eqnarray}
    \begin{aligned}
        \mathbb{M}_0^n[u](\mathbf{x}) =&\mathbb{E}\left[\left(u(\mathbf{x},\bGamma_{\epsilon\bomega}) - \mathbb{E}[u](\mathbf{x})\right)^n\right]\\
        =&\epsilon^{n+1}\binom{n}{1}\mathbb{E}\left[\left(\sum_{i=1}^m\omega_i\delta_{\mathbf{v}_i}u\right)^{n-1}\left(\frac{1}{2!}\sum_{i, j = 1}^{m}\omega_i\omega_j\delta_{[\mathbf{v}_i,\mathbf{v}_j]}u(\mathbf{x})\right.\right.\\
        &\left.\left.- \frac{1}{3!}\sum_{i = 1}^m\delta_{[\mathbf{v}_i,\mathbf{v}_i]}u(\mathbf{x})\right)\right]\\
        &+\mathcal{O}\left(\epsilon^{n+3}\right).
    \end{aligned}
\end{eqnarray}
\end{theorem}
%%%%%%%%%%%%%%%%%%%%%%%%%%%%%%%%%%% N, m
%which implies that using the first order expansion is sufficient to estimate the variance with the error of $\mathcal{O}\left(\epsilon^4\right)$.

Similar to the case for higher order moments, the higher order shape Taylor expansion can yield a more accurate estimation for central moments. In particular, let $(n,N)$ be a pair of positive integers satisfying $n+N$ is odd. When using the $N$th order expansion to estimate the $n$th order central moment of a random field $u$, the approximation error is $\mathcal{O}(\epsilon^{n+N+1})$.

According to the conclusions above, one can see that the accuracy of estimating the moments $\mathbb{M}^n[u]$ is mainly determined by the order of the shape Taylor expansion. For the central moments $\mathbb{M}_0^n[u]$, the accuracy depends on both the moment order $n$ and the expansion order $N$.

\section{Numerical examples}\label{Numerical}
% Here we set $n = 2$ in the numerical experiments.
The numerical examples in this section illustrate the properties of the shape Taylor expansion and its applications in uncertainty quantification. Here we employ the boundary integral method~\cite{kress2018inverse} to solve the scattering problems with two types of incident waves as given by equations~\eqref{incwave1} and~\eqref{incwave2}. The first part of this section investigates the approximation properties of the shape Taylor expansion in the scattering problems with sound-hard, sound-soft, and transmission boundary conditions. The second part applies the shape Taylor expansion to estimate the moments of an impedance scattering problem with randomly perturbed boundary.

\subsection{Numerical implementation}
In this subsection, we briefly introduce the boundary integral equation with Nystr{\"o}m discretization, based on trigonometric interpolation~\cite{colton2019inverse}, to obtain the scattered field and its shape derivatives numerically. 

For the sound-soft scattering problems, the scattered field $u$ is represented by the combined layer potential
\begin{eqnarray}\label{Inteq1}
    u(\mathbf{x}) = \left(\mathbf{D}\rho\right)(\mathbf{x}) - ik\left(\mathbf{S}\rho\right)(\mathbf{x}),\quad\mathbf{x}\in D_{\rm ex}.
\end{eqnarray}
For the sound-hard scattering problems, $u$ is represented by the single layer potential
\begin{eqnarray}\label{Inteq2}
     u(\mathbf{x})=\left(\mathbf{S}\rho\right)(\mathbf{x}),\quad\mathbf{x}\in D_{\rm ex}.
\end{eqnarray}
For the impedance and transmission problems, $u$ is given by Green's representation formula
\begin{eqnarray}\label{Inteq3}
    u(\mathbf{x}) = \left(\mathbf{D}u\right)(\mathbf{x})- \left(\mathbf{S}\frac{\partial u}{\partial\mathbf{n}}\right)(\mathbf{x}),\quad\mathbf{x}\in D_{\rm ex}.
\end{eqnarray}
The representations of $u$ in equations~\eqref{Inteq1},~\eqref{Inteq2} and~\eqref{Inteq3} lead to a second kind integral equation (or a second kind integral system for transmission problems under representation~\eqref{Inteq3}) on the boundary, which is given in the form of
\begin{eqnarray}\label{IntEq}
    C\rho(s) + \int_{\bGamma}K(t,s)\rho(t)dt = f(s),
\end{eqnarray}
where $\rho$ is the unknown density function, $C$ is a nonzero constant, $K$ is the singular integral kernel and $f$ is a given function that depends on the incident wave. An equidistant mesh with nodes $t_j = j\Delta t$ for $j = 1,\dots,N_t$ is introduced to discrete equation~\eqref{IntEq}. We choose $N_t = 400$ in the following examples. Based on the trigonometric interpolation~\cite{colton2019inverse}, we can solve for $\rho_j:=\rho(t_j)$ and compute the scattered field and its corresponding shape derivatives with spectral accuracy. 

 In Subsections 5.2-5.4, to compare the error between the shape Taylor expansion of $u(\cdot,\bGamma)$ and the perturbed field $u(\cdot,\bGamma_{\rm per})$, we consider the case where $\bGamma$ is perturbed by a single velocity field $\mathbf{v}$ with varying perturbation magnitudes $\epsilon$, in which the perturbed boundary is denoted as $\bGamma_{\epsilon,\mathbf{v}}$. A set of observation points $\mathbf{x}_1,\mathbf{x}_2,...,\mathbf{x}_M$ is placed on a circle $\bGamma_R$ with radius $R>0$. The error is quantified by the residual function defined as
\begin{eqnarray}
    \mathbf{Res}(\epsilon,N): = \left(\sum_{j = 1}^M|u(\mathbf{x}_j,\bGamma_{\epsilon,\mathbf{v}}) -  \mathbf{Taylor}(\mathbf{x}_j;u,\mathbf{v},N)|\right)^{\frac{1}{N+1}}.
\end{eqnarray}
The residual is determined by the perturbation magnitude $\epsilon$ and the order of Taylor expansion $N$. Theoretically, according to equation~\eqref{shapeTaylorMultiV}, $\mathbf{Res}(\epsilon,N)$ is of order $\mathcal{O}(\epsilon)$.

\subsection{Sound-hard boundary}\label{Numerical11}

\begin{figure}[htbp]
    \centering
    \begin{subfigure}[b]{0.4\textwidth}
        \centering
        \includegraphics[width=\textwidth]{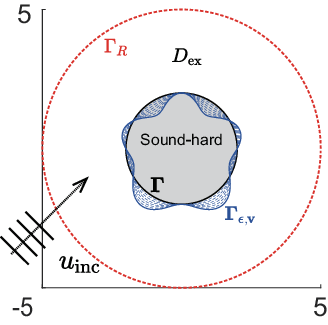}
        \caption{Geometry of the scattering problem with a perturbed boundary.}
        \label{shard1}
    \end{subfigure}

     %\vspace{0.4cm}

     \begin{subfigure}[b]{0.47\textwidth}
        \centering
        \includegraphics[width=\textwidth]{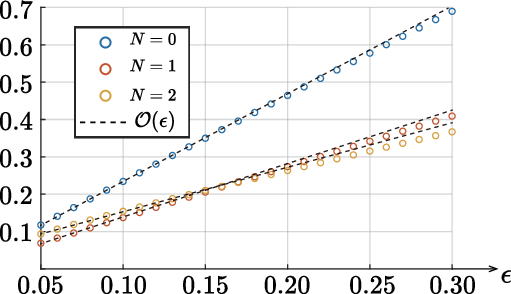}
        \caption{The residual of shape Taylor expansion of different orders at $k = 3$.}
        \label{shard2}
     \end{subfigure}
     \hfill
     \begin{subfigure}[b]{0.47\textwidth}
         \centering
        \includegraphics[width=\textwidth]{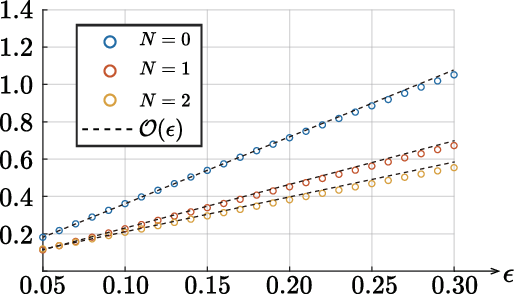}
        \caption{The residual of shape Taylor expansion of different orders at $k = 5$.}
        \label{shard3}
     \end{subfigure}
     \caption{Results for the scattering of a perturbed circle with sound-hard boundary condition.}
     \label{shard}
\end{figure}

 Let $\bGamma$ be a circle with radius $r$, the parametrization of $\bGamma$ using the arc length $s$ is given by
\begin{eqnarray}
    \bgamma(s) = r\left[\cos\frac{s}{r},\sin\frac{s}{r}\right]^\top,\qquad s\in[0,2\pi r].
\end{eqnarray}
The velocity field restricted to $\bGamma$ for perturbation is given by $\mathbf{v}(s) = v(s)\mathbf{n}(s)$ with
\begin{eqnarray}
    v(s) = 0.4\sin\frac{2s}{r}\cos\frac{3s}{r}.
\end{eqnarray}
The incident field $\phi$ is generated by a plane wave~\eqref{incwave1} with propagating direction $\mathbf{z} = [1,1]^\top/\sqrt{2}$. We test the cases for different wave numbers. Let the radii of $\bGamma$ and $\bGamma_R$ be $r = 2$ and $R = 5$, respectively. The perturbation magnitude $\epsilon$ is constrained by $\epsilon\le 0.3$. The highest order of Taylor expansion is tested at $N=2$, which implies that, at most, the third order normal derivatives of the total field must be computed. 

The result of the boundary perturbation is shown in Figure~\ref{shard1}, where $\bGamma_{\epsilon,\mathbf{v}}$ gradually transforms into a general star-shaped geometry as $\epsilon$ increases. The comparisons between the perturbation magnitude $\epsilon$ and the residuals $\mathbf{Res}(\cdot, N)$ for $N = 0, 1, 2$ are shown in Figures~\ref{shard2} and \ref{shard3}, corresponding to wavenumbers $k = 3$ and $k = 5$, respectively. The plots of $\epsilon$ versus $\mathbf{Res}(\cdot, N)$ show an approximately linear relation between them, especially when the perturbation magnitude is small. This relationship demonstrates that the error between the $N$th order shape Taylor expansion of $u(\cdot,\bGamma)$ and the perturbed field $u(\cdot,\bGamma_{\epsilon,\mathbf{v}})$ is equal to $\mathcal{O}(\epsilon^{N+1})$. Furthermore, comparing the results for $k = 3$ and $k = 5$ suggests that the approximation error increases with the frequency of the incident field. To further verify this, we choose an observation point $\mathbf{x}_{\rm obs} = [0,4]^\top$ and present the relative errors of the second order shape Taylor expansion for different $k$ and $\epsilon$ in Table~\ref{tab_soundhard}.
\begin{table}[ht]
    \centering
    \begin{tabular}{|c|c|c|c|c|}
    \hline
     $N=2$    & $k = 3$ & $k=\pi$ & $k = 5$ & $k=2\pi$ \\
    \hline
     $\epsilon = 0.25$    & 5.5341e-03 &6.2543e-03&1.0349e-02&1.0393e-02\\
     $\epsilon = 0.20$    &3.5169e-03&3.9432e-03&6.3956e-03&9.1226e-03\\
     $\epsilon = 0.15$   &1.9648e-03&2.1877e-03&3.4899e-03&6.5305e-03\\
     $\epsilon = 0.10$    &8.6731e-04&9.5972e-04&1.5142e-03&3.5201e-03\\
    \hline
    \end{tabular}
    \vspace{3mm}
    \caption{Relative errors of the second order shape Taylor expansion for the sound-hard scattering problem with different boundary perturbations and wavenumbers.}
    \label{tab_soundhard}
\end{table}

\subsection{Sound-soft boundary}

\begin{figure}[htbp]
    \centering
    \begin{subfigure}[b]{0.4\textwidth}
        \centering
        \includegraphics[width=\textwidth]{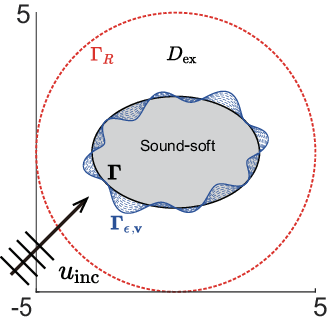}
        \caption{Geometry of the scattering problem with a perturbed boundary.}
        \label{Ssoft1}
    \end{subfigure}

     \begin{subfigure}[b]{0.47\textwidth}
        \centering
        \includegraphics[width=\textwidth]{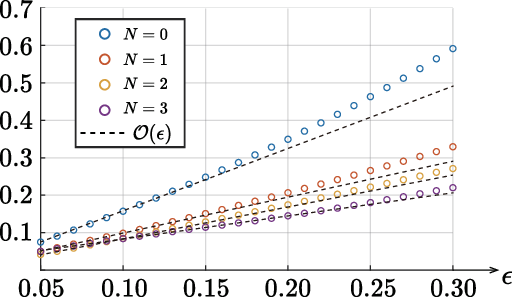}
        \caption{The residual of shape Taylor expansion of different orders at $k = 3$.}
        \label{Ssoft2}
     \end{subfigure}
     \hfill
     \begin{subfigure}[b]{0.47\textwidth}
         \centering
        \includegraphics[width=\textwidth]{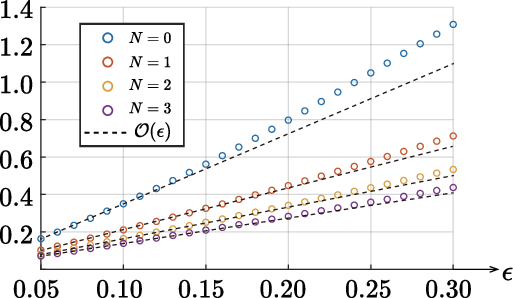}
        \caption{The residual of shape Taylor expansion of different orders at $k = 5$.}
        \label{Ssoft3}
     \end{subfigure}
     \caption{Results for the scattering of a perturbed ellipse with sound-soft boundary condition.}
     \label{Ssoft}
    
\end{figure}

 Let $\bGamma$ be an ellipse parameterized by
\begin{eqnarray}
    \bgamma(\theta) = \left[a\cos\theta,b\sin\theta\right]^\top,\qquad\theta\in[0,2\pi].
\end{eqnarray}
The velocity field generating the perturbed boundary $\bGamma_{\epsilon,\mathbf{v}}$ is given by
\begin{eqnarray}
   \mathbf{v}(\theta) =  v(\theta)\mathbf{n}(\theta) = 0.4\sin(5\theta)\cos(3\theta)\mathbf{n}(\theta).
\end{eqnarray}
Note that the boundary parameter $\theta$ is not the arc length parameter $s$, so the tangential differential operator should be replaced by $\partial_s = |\partial_\theta\bgamma|^{-1}\partial_\theta$. To compare the error between the shape Taylor expansion of $u(\cdot,\bGamma)$ and the perturbed field $u(\cdot,\bGamma_{\epsilon,\mathbf{v}})$, we choose the observation points $\mathbf{x}_1,\dots,\mathbf{x}_M$ as described in Subsection~\ref{Numerical11}. We still set the perturbation magnitude $\epsilon\le0.3$, and the propagating direction of the incident wave as $\mathbf{z} = [1,1]^\top/\sqrt{2}$. The semi-axes of the ellipse boundary $\bGamma$ are given by $a = 3$ and $b = 2$. The highest order of the Taylor expansion is tested at $N = 3$. According to equation \eqref{vecBDN1}, this requires computing up to the third order normal derivative of the scattered field on $\bGamma$.

The result of the sound-soft boundary perturbation is shown in Figure~\ref{Ssoft1}, where the elliptical boundary $\bGamma$ gradually evolves into a complex shape $\bGamma_{\epsilon,\mathbf{v}}$ as the magnitude $\epsilon$ increases. Figures~\ref{Ssoft2} and~\ref{Ssoft3} give the plot of the residual $\mathbf{Res}(\cdot, N)$ and the magnitude $\epsilon$ for wavenumbers $k = 3$ and $k=5$, respectively. Similar to the sound-hard case, the residual is slightly larger when $k=5$. When $\epsilon$ is small, $\mathbf{Res}(\cdot,N)$ holds an approximately linear relationship with $\epsilon$. As $\epsilon$ increases, the approximation error of the first order expansion grows significantly, while the higher order expansions, especially for $N = 3$, still provide reliable approximations. Table~\ref{tab_soundsoft} gives the relative errors of the third order shape Taylor expansion at $\mathbf{x}_{\rm obs} = [0,4]^\top$ for different wavenumbers.
\begin{table}[ht]
    \centering
    \begin{tabular}{|c|c|c|c|c|}
    \hline
     $N=3$    & $k = 3$ & $k=\pi$ & $k = 5$ & $k=2\pi$ \\
    \hline
     $\epsilon = 0.25$    &1.0663e-04&1.1065e-04&1.8471e-04&2.2798e-04\\
     $\epsilon = 0.20$    &5.1302e-05&5.3292e-05&8.9754e-05&1.1245e-04\\
     $\epsilon = 0.15$   &2.0257e-05&2.1069e-05&3.5878e-05&4.5713e-05\\
     $\epsilon = 0.10$    &5.5933e-06&5.8263e-06&1.0057e-05&1.3059e-05\\
    \hline
    \end{tabular}
    \vspace{3mm}
    \caption{Relative errors of the third order shape Taylor expansion for the sound-soft scattering problem with different boundary perturbations and wavenumbers.}
    \label{tab_soundsoft}
\end{table}

\subsection{Penetrable scattering}
\begin{figure}[htbp]
    \centering
    \begin{subfigure}[b]{0.55\textwidth}
        \centering
        \includegraphics[width=\textwidth]{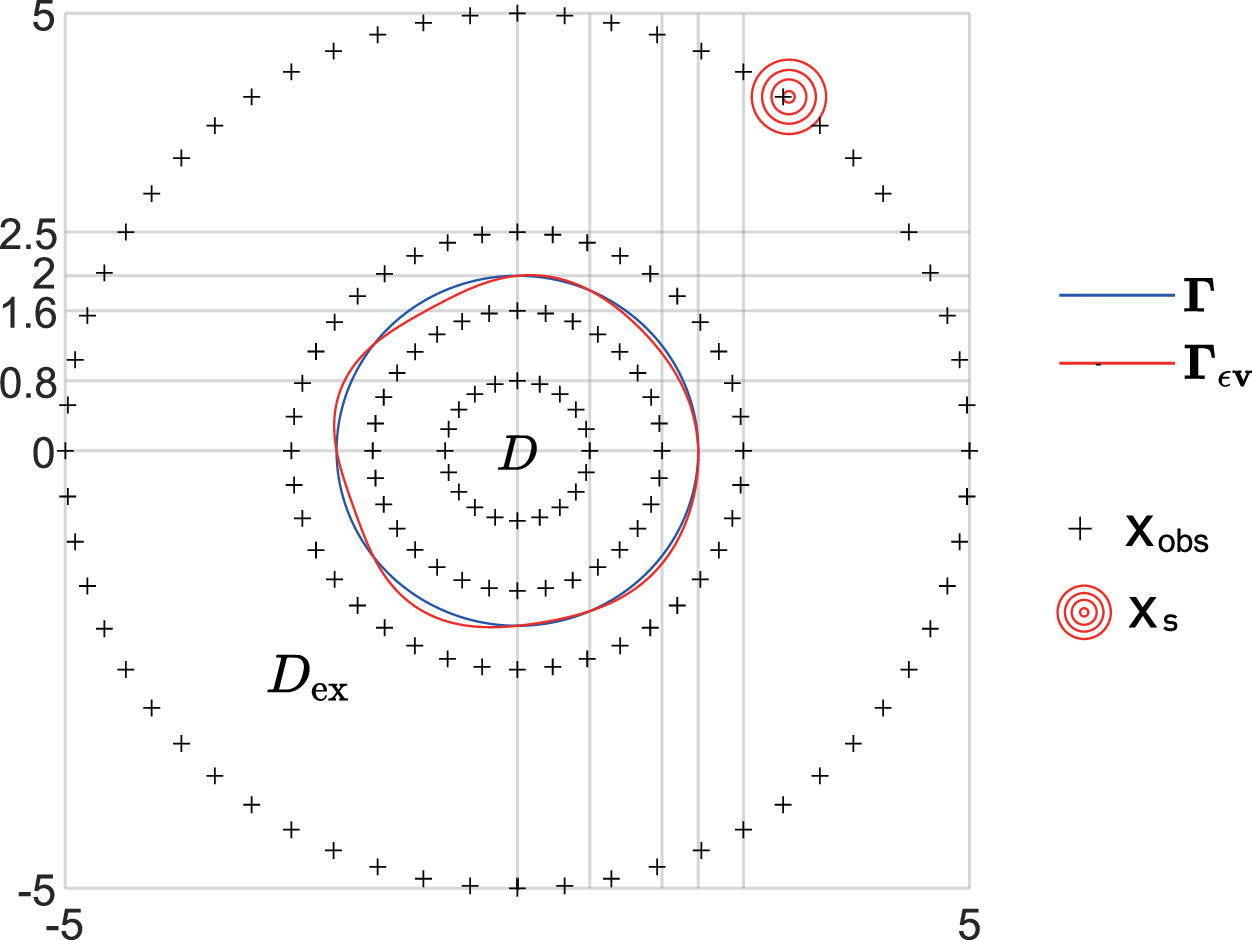}
        \caption{Geometry of the boundary perturbation, source point and observation points.}
        \label{Tran1}
    \end{subfigure}
    \hfill
    \begin{subfigure}[b]{0.4\textwidth}
        \centering
        \includegraphics[width=\textwidth]{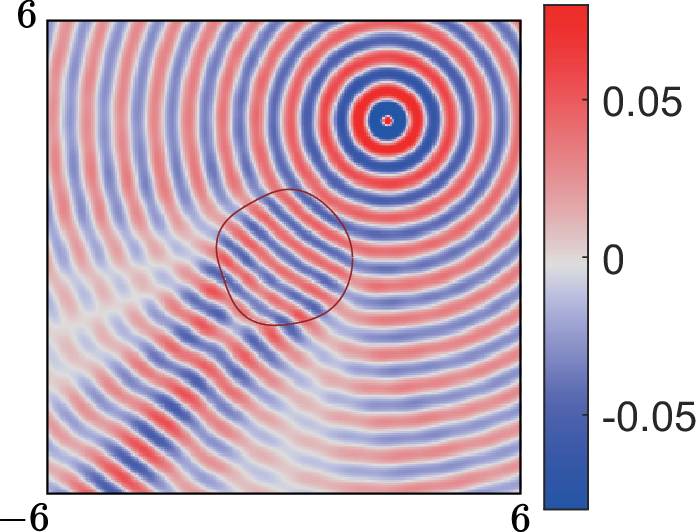}
        \caption{The total field of a perturbed circle.}
        \label{Tran2a}
    \end{subfigure}

     \begin{subfigure}[b]{0.32\textwidth}
         \centering
        \includegraphics[width=\textwidth]{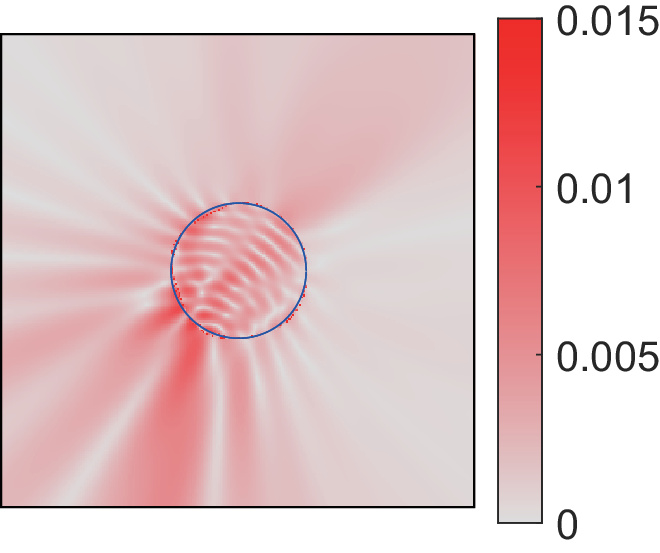}
        \caption{Approximation error of the shape Taylor expansion of order $N=0$.}
        \label{Tran2b}
     \end{subfigure}
     \hfill
     \begin{subfigure}[b]{0.32\textwidth}
         \centering
        \includegraphics[width=\textwidth]{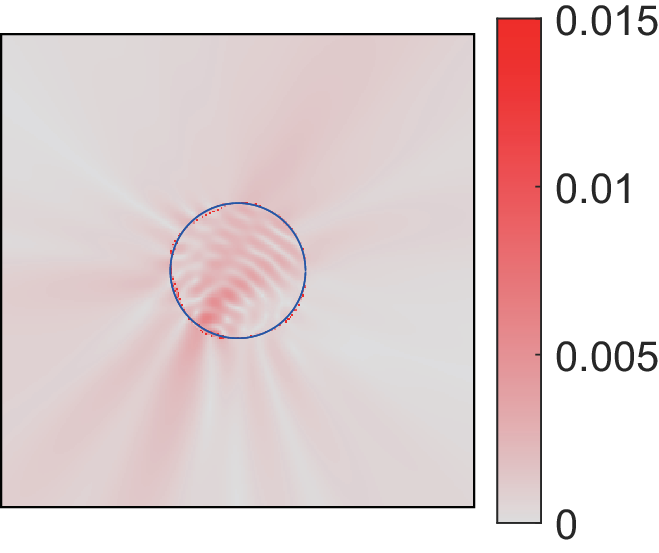}
        \caption{Approximation error of the shape Taylor expansion of order $N=1$.}
        \label{Tran2c}
     \end{subfigure}
     \hfill
     \begin{subfigure}[b]{0.32\textwidth}
         \centering
        \includegraphics[width=\textwidth]{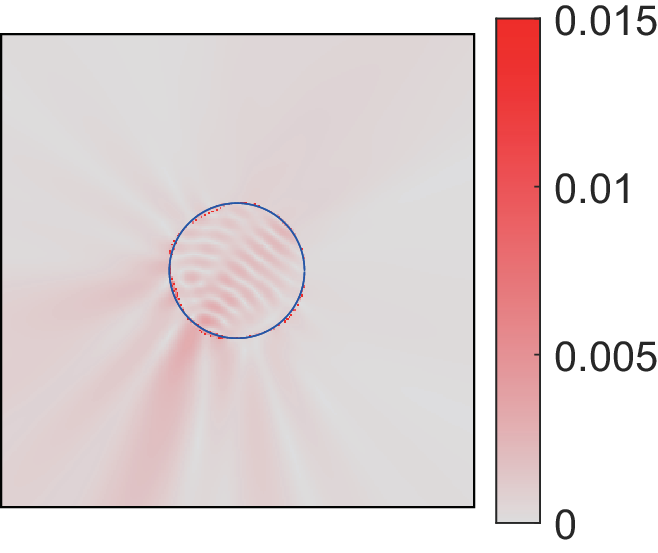}
        \caption{Approximation error of the shape Taylor expansion of order $N=2$.}
        \label{Tran2d}
     \end{subfigure}

     \caption{Illustration for the scattering of a perturbed circle with transmission boundary condition.}
     \label{Tran12}
\end{figure}

\begin{figure}[htbp]
    \centering
        \begin{subfigure}[b]{0.47\textwidth}
         \centering
        \includegraphics[width=\textwidth]{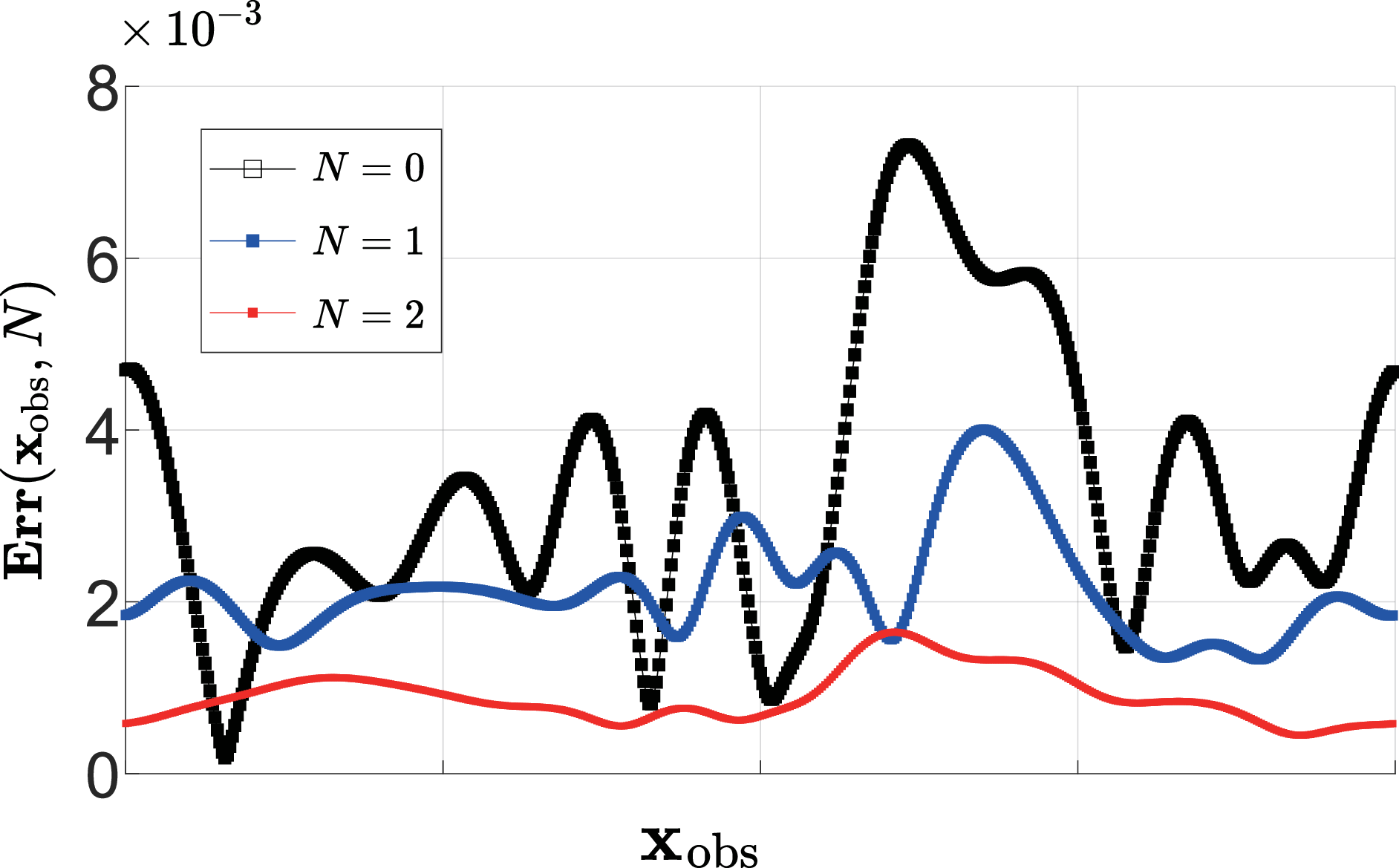}
        \caption{Error at $r = 0.8$.}
        \label{Tran3d}
     \end{subfigure}
\hfill
     \begin{subfigure}[b]{0.47\textwidth}
         \centering
        \includegraphics[width=\textwidth]{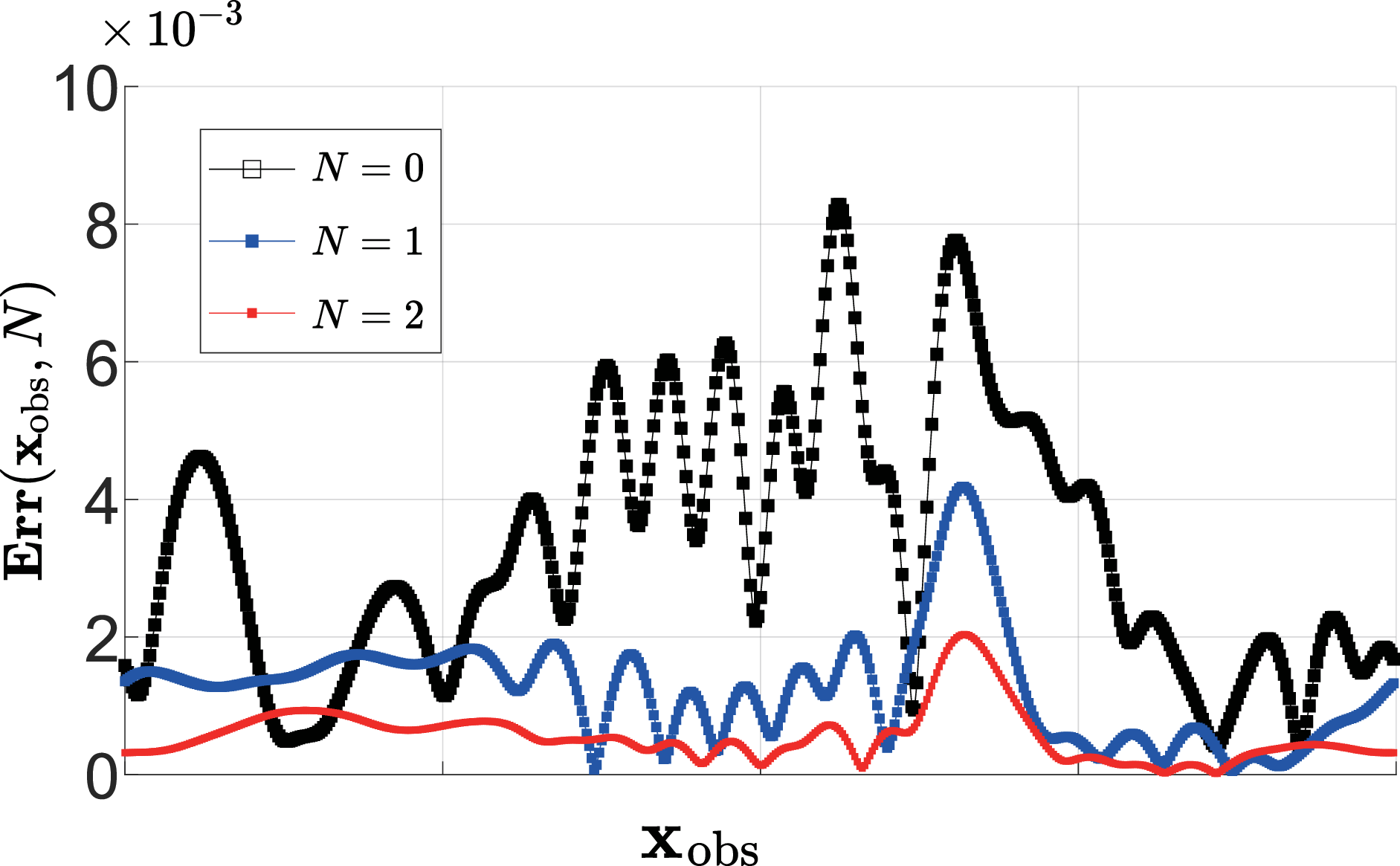}
        \caption{Error at $r = 1.6$.}
        \label{Tran3c}
     \end{subfigure}
       \vspace{0.4cm}
       
    \begin{subfigure}[b]{0.47\textwidth}
        \centering
        \includegraphics[width=\textwidth]{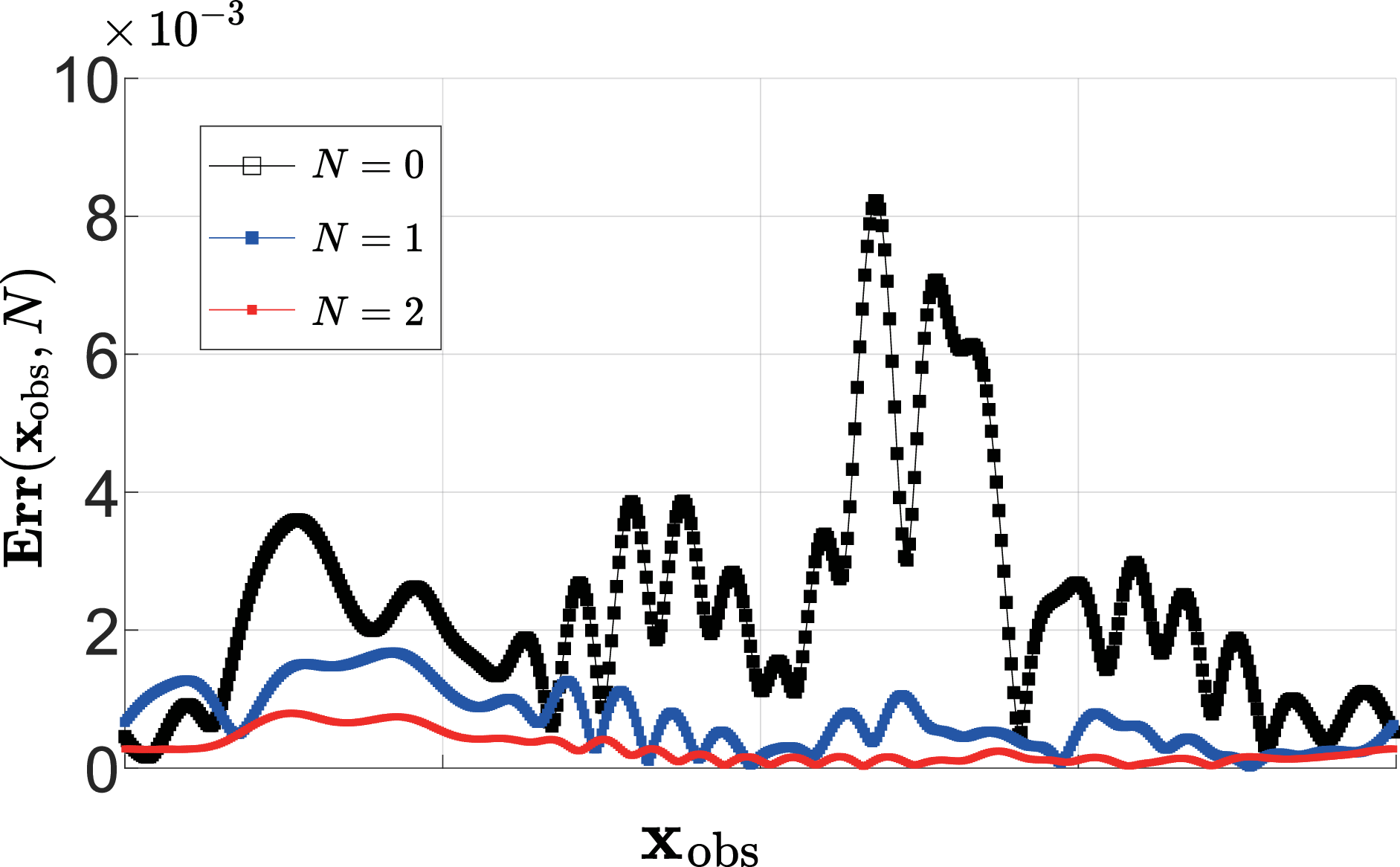}
        \caption{Error at $r = 2.5$.}
        \label{Tran3a}
    \end{subfigure}
    \hfill
    \begin{subfigure}[b]{0.47\textwidth}
        \centering
        \includegraphics[width=\textwidth]{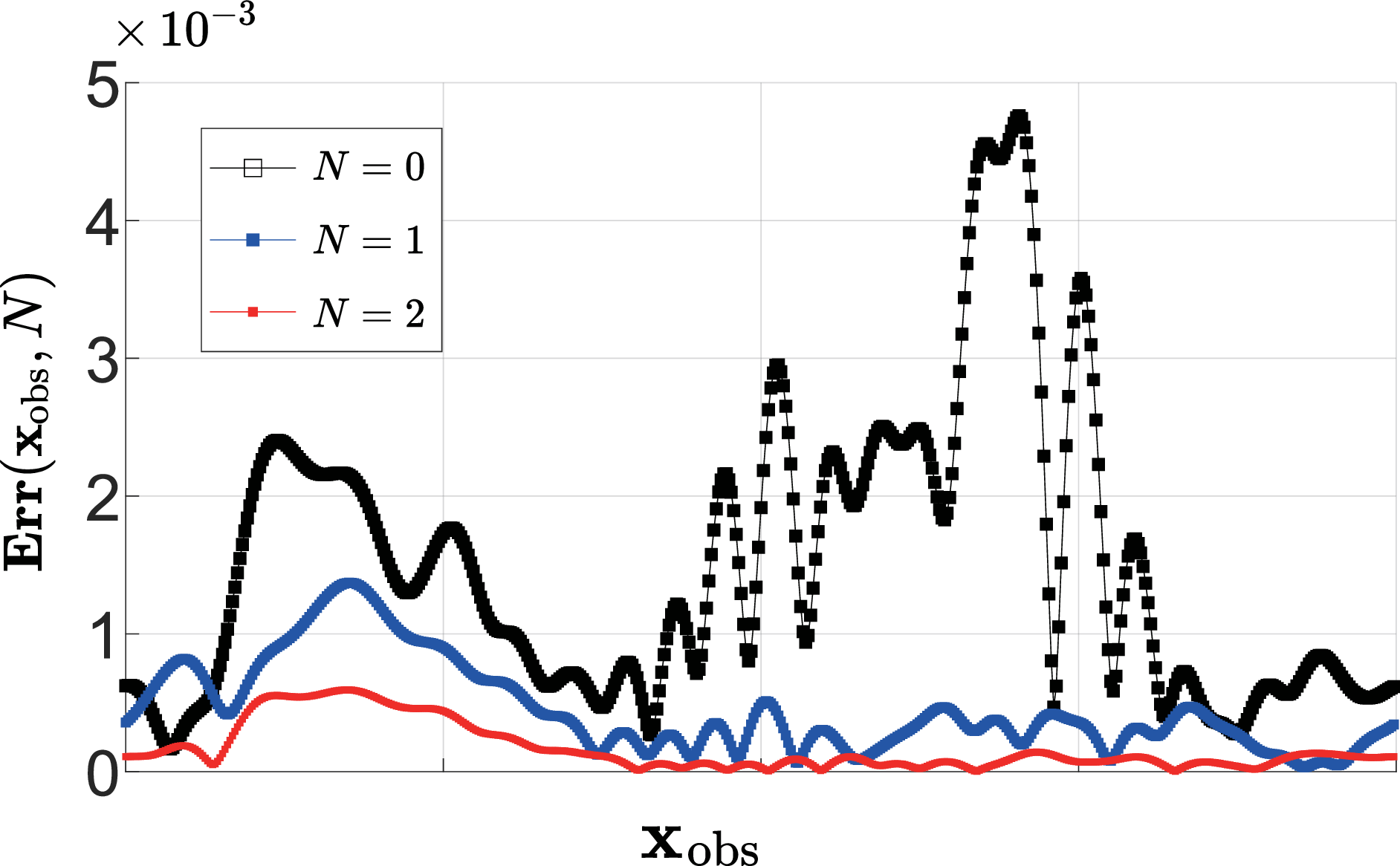}
        \caption{Error at $r = 5$.}
        \label{Tran3b}
    \end{subfigure}

     \caption{Approximation errors of the shape Taylor expansion for the transmission scattering problem at observation points with different radii.}
     \label{Tran3}
\end{figure}

In this example, the total field of this scattering problem is defined as
\begin{eqnarray}
    u = \left\{\begin{aligned}
        &u_{\rm in},&&\mathbf{x}\in D,\\
        &u_{\rm ex} + \phi,&&\mathbf{x}\in D_{\rm ex}.
    \end{aligned}\right.
\end{eqnarray}
 We choose the boundary $\bGamma$ as a circle with radius $r = 2$. Here, we denote the scattered field inside and outside the scatterer as $u_{\rm in}$ and $u_{\rm ex}$, respectively. Recalling equation~\eqref{scattering_eqn1}, the transmission conditions across the medium boundary $\bGamma$ for $u$ are given by
\begin{eqnarray}
\left\{
    \begin{aligned}
        &\lim_{\mathbf{x}\rightarrow\bGamma^{+}}u_{\rm ex} - \lim_{\mathbf{x}\rightarrow\bGamma^{-}}u_{\rm in} = \phi,\\
        &\lim_{\mathbf{x}\rightarrow\bGamma^{+}}\alpha_{\rm ex}\partial_{\mathbf{n}}u_{\rm ex} - \lim_{\mathbf{x}\rightarrow\bGamma^{-}}\alpha_{\rm in}\partial_{\mathbf{n}}u_{\rm in} = \alpha_{\rm ex}\partial_{\mathbf{n}}\phi,
    \end{aligned}\right.
\end{eqnarray}
where the medium parameters are set as $\alpha_{\rm in} = 0.7$ and $\alpha_{\rm ex} = 1$. The velocity field $\mathbf{v}$ is given by
\begin{eqnarray}
    \mathbf{v}(s) = v(s)\mathbf{n}(s) = 0.25\left(\sin\frac{2s}{r}\cos\frac{3s}{r} - 0.7\sin\frac{4s}{r}\right)\mathbf{n}(s),
\end{eqnarray}
where $s$ is the arc length parameter. A point source located at $\mathbf{x}_{\rm s} = [3,4]^\top$ with $k = 2\pi$ generates the incident field defined by equations~\eqref{incwave1} and~\eqref{incwave2}. To compare the errors between the Taylor expansion of $u(\cdot,\bGamma)$ and the perturbed field $u(\cdot,\bGamma_{\epsilon,\mathbf{v}})$ over the whole domain, we place the observation points $\mathbf{x}_{\rm obs}$ in both $D$ and $D_{\rm ex}$. Especially, the observation points are distributed on four concentric circles with radii of $0.8$, $1.6$, $2.5$, and $5$, respectively. Figures~\ref{Tran1} and \ref{Tran2a} show the setup of the scattering problem and the total field of a perturbed circle, respectively.

We let the perturbation magnitude $\epsilon = 0.3$ and compute the shape Taylor expansion for $N = 1$ and $2$. The error between the expansion and the perturbed field in the domain $( D\cap D^{\rm per}_{\rm in})\cup( D_{\rm ex}\cap D^{\rm per}_{\rm ex})$ is shown in Figure~\ref{Tran2b}--\ref{Tran2d}. Figure~\ref{Tran3} shows the absolute errors at the observation points. It can be seen that the second order shape Taylor expansion can reliably approximate $u(\cdot,\bGamma_{\epsilon,\mathbf{v}})$ even with a relatively large perturbation, and the approximation performs consistently well inside and outside the scatterer. To further verify this, we also list the relative errors of the second order shape Taylor expansion at $\mathbf{x}_{\rm obs} = [1,0]^\top\in D$ and $\mathbf{x}_{\rm obs} = [3,3]^\top\in D_{\rm ex}$ in Tables~\ref{tab_tranin} and~\ref{tab_tranex}, respectively.

\begin{table}[ht]
    \centering
    \begin{tabular}{|c|c|c|c|c|}
    \hline
     $N=2$    & $k = 3$ & $k=\pi$ & $k = 5$ & $k=2\pi$ \\
    \hline
     $\epsilon = 0.25$    &7.5869e-04&7.7379e-04&3.5713e-03&6.0900e-03\\
     $\epsilon = 0.20$    &4.0084e-04&4.0752e-04&1.7875e-03&3.1783e-03\\
     $\epsilon = 0.15$   &1.7856e-04&1.8122e-04&7.2236e-04&1.3633e-03\\
     $\epsilon = 0.10$    &5.8924e-05&5.9925e-05&1.9412e-04&4.0488e-04\\
    \hline
    \end{tabular}
    \vspace{3mm}
    \caption{Relative error of the second order shape Taylor expansion at the observation point inside the penetrable scatterer.}
    \label{tab_tranin}
\end{table}

\begin{table}[ht]
    \centering
    \begin{tabular}{|c|c|c|c|c|}
    \hline
      $N=2$   & $k = 3$ & $k=\pi$ & $k = 5$ & $k=2\pi$ \\
    \hline
     $\epsilon = 0.25$    &1.1782e-02&1.0999e-02&6.1792e-02&1.1954e-01\\
     $\epsilon = 0.20$    &5.9968e-03&5.5827e-03&3.1639e-02&6.1621e-02\\
     $\epsilon = 0.15$   &2.5212e-03&2.3344e-03&1.3332e-02&2.6114e-02\\
     $\epsilon = 0.10$    &7.4958e-04&6.8593e-04&3.9455e-03&7.7662e-03\\
    \hline
    \end{tabular}
    \vspace{3mm}
    \caption{Relative error of the second order shape Taylor expansion at the observation point outside the penetrable scatterer.}
    \label{tab_tranex}
\end{table}

\subsection{Moment estimations of a random field}\label{Randomproblem2}

\begin{figure}
    \centering
    \includegraphics[width=0.5\linewidth]{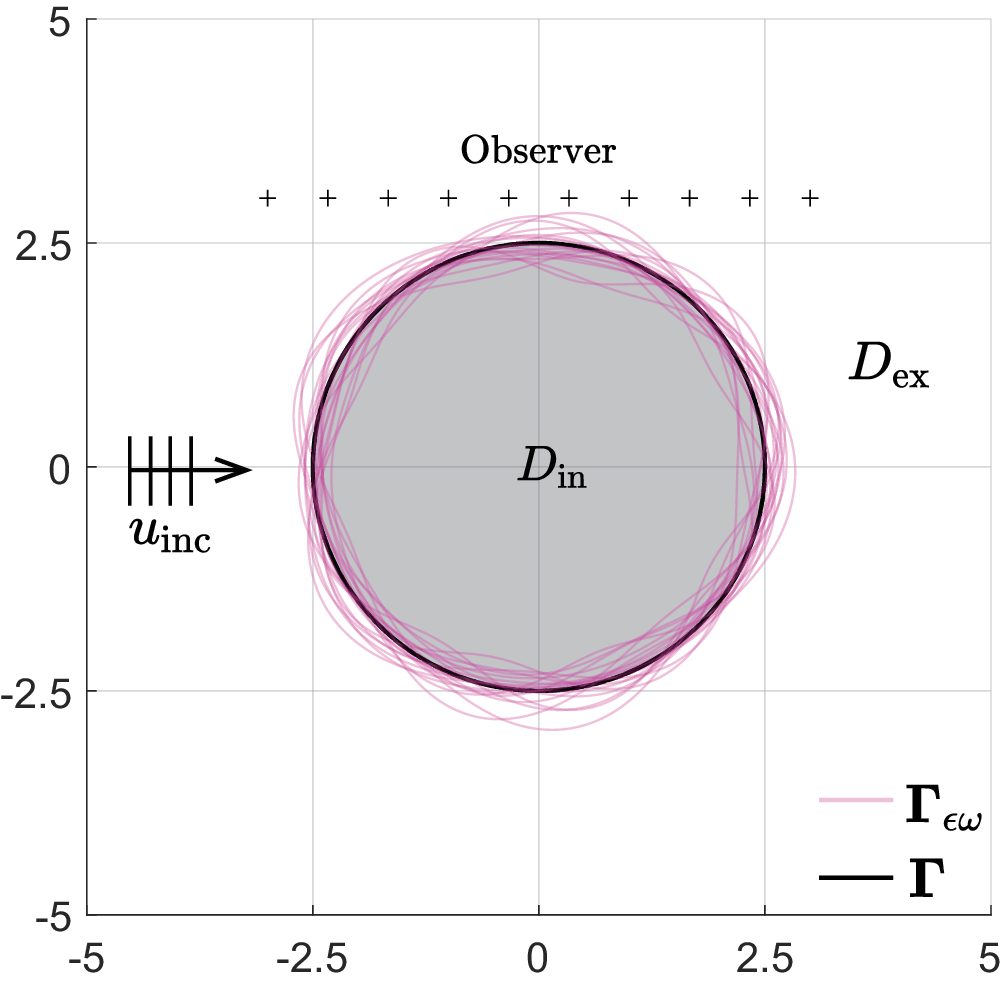}
    \caption{Scattering  with a random boundary perturbation under the impedance boundary condition}
    \label{Rand1}
\end{figure}

We consider the random scattering problems with the impedance boundary condition. The randomly perturbed boundary $\bGamma_{\epsilon\bomega}$ is given by equation~\eqref{randbound}, where the reference boundary $\bGamma$ is a circle with radius $r = 2.5$. We assume the perturbation in equation~\eqref{randbound} is composed of 11 velocity fields, which are given by $\mathbf{v}_j(\theta) = v_j(\theta)\mathbf{n}(\theta)$ with
\begin{eqnarray}
\left\{
\begin{aligned}
    &v_1(\theta) = 1,\\
    &v_j(\theta) = \cos(j-1)\theta\quad{\rm for}\quad j = 2,...,6,\\
    &v_j(\theta) = \sin(j-6)\theta\quad{\rm for}\quad j = 7,...,11.
\end{aligned}\right.
\end{eqnarray}
 The perturbation magnitude is set as $\epsilon\le 0.03$. Ten observation points are placed in the near field region outside the scatterer. The geometry of the random boundary perturbation is given in Figure~\ref{Rand1}.

To investigate the estimation error, we generate the reference data of $\mathbb{M}^n[u]$ by performing Monte Carlo simulations 3000 times as detailed in~\cite{harbrecht2013first}. According to Section~\ref{Randomproblem}, the moment estimations $\mathcal{E}^n_N$ based on the $N$th order shape Taylor expansion for $N = 0,1,2$ are given by
\begin{eqnarray}\label{est012}
    \begin{aligned}
        \mathcal{E}^n_0(\mathbf{x}):=& u^n(\mathbf{x}),\\
        \mathcal{E}^n_1(\mathbf{x}): =& u^n(\mathbf{x}) + \frac{\epsilon^2}{3}\sum_{i = 1}^{11}\binom{n}{2}u^{n-2}(\mathbf{x})\delta_{\mathbf{v}_i}u^2(\mathbf{x}),\\
        \mathcal{E}^n_2(\mathbf{x}): =&u^n(\mathbf{x}) + \frac{\epsilon^2}{3}\sum_{i = 1}^{11}\left(\binom{n}{2}u^{n-2}(\mathbf{x})\delta_{\mathbf{v}_i}u^2(\mathbf{x}) + \binom{n}{1}u^{n-1}(\mathbf{x})\frac{1}{2}\delta_{[\mathbf{v}_i,\mathbf{v}_i]}u(\mathbf{x})\right).
    \end{aligned}
\end{eqnarray}
In the uncertainty quantification of wave scattering problems without using higher order shape expansions~\cite{hao2018computation}, $\mathcal{E}^1_0 = \mathcal{E}^1_1 = u(\cdot, \bGamma)$ always represents the expectation of the scattered field. However, based on the higher order shape Taylor expansion, one can see that the field $u(\cdot, \bGamma)$ from an obstacle with the expected shape $\bGamma$ is generally not equal to the expectation of the field $\mathbb{E}[u]$ (see Figure~\ref{Rand3} for the case of $n=1$). To compare the error between the estimation in equation~\eqref{est012} and the reference $\mathbb{M}^n[u]$, let us consider the absolute residual:
\begin{eqnarray}
    \mathbf{Res}(\mathcal{E}^n) = \sum_{\mathbf{x}_{\rm obs}}|\mathbb{M}^n[u](\mathbf{x}_{\rm obs}) - \mathcal{E}^n(\mathbf{x}_{\rm obs})|,
\end{eqnarray}
where $\mathbf{x}_{\rm obs}$ are a set of observation points in the exterior region $D_{\rm ex}$, as shown in Figure~\ref{Rand1}.

\begin{figure}[htbp]
    \centering
    \begin{subfigure}[b]{0.32\linewidth}
    \centering
    \includegraphics[width=\textwidth]{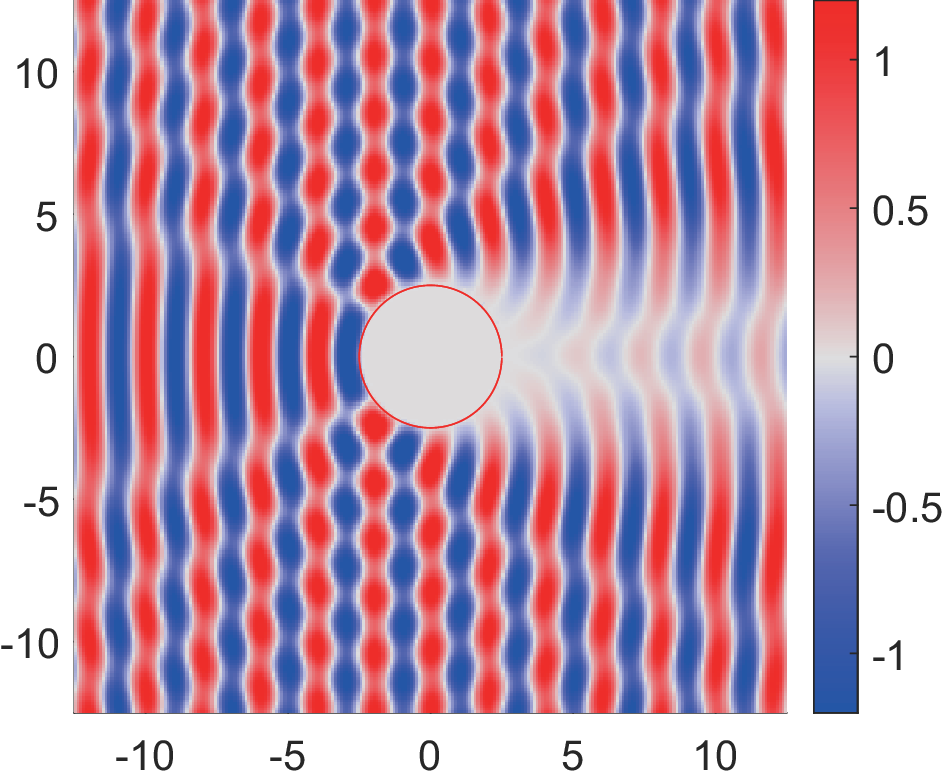}
    \caption{$\mathbf{Re}\left(\mathbb{M}^1[u]\right)$}
    \label{Rand2a}
    \end{subfigure}
    \hfill
    \begin{subfigure}[b]{0.32\linewidth}
    \centering
    \includegraphics[width=\textwidth]{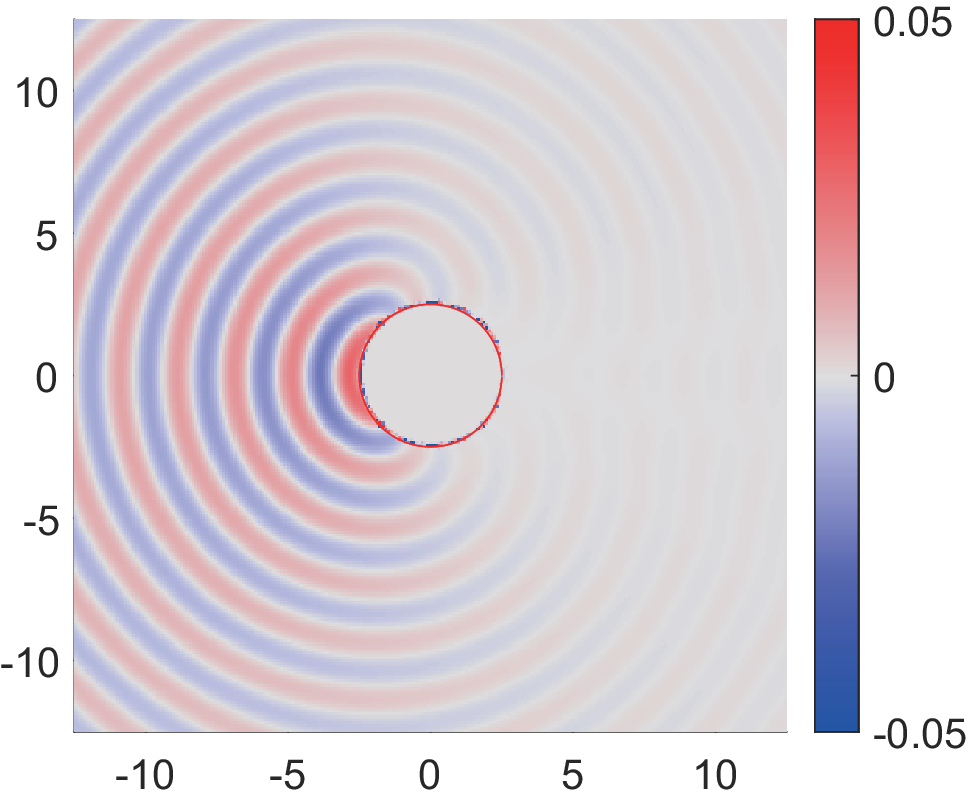}
    \caption{$\mathbf{Re}\left(\mathbb{M}^1[u]- \mathcal{E}^1_1\right)$}
    \label{Rand2b}
    \end{subfigure}
    \hfill
    \begin{subfigure}[b]{0.32\linewidth}
    \centering
    \includegraphics[width=\textwidth]{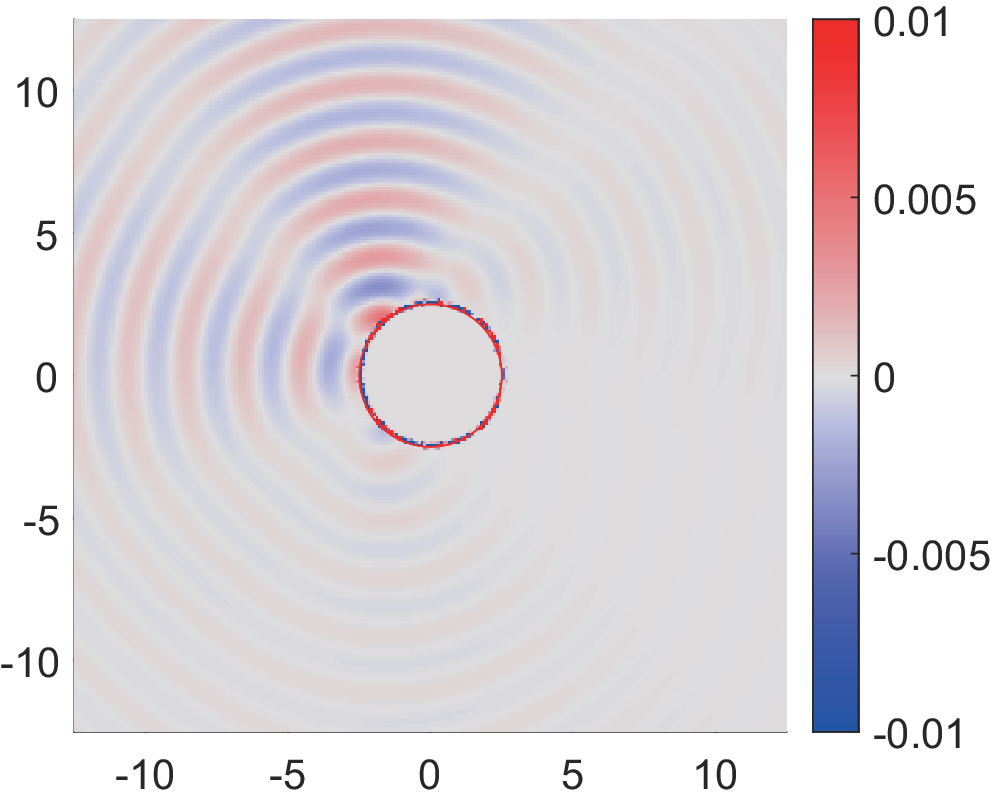}
    \caption{$\mathbf{Re}\left(\mathbb{M}^1[u]- \mathcal{E}^1_2\right)$}
    \label{Rand2c}
    \end{subfigure}

\begin{subfigure}[b]{0.32\linewidth}
    \centering
    \includegraphics[width=\textwidth]{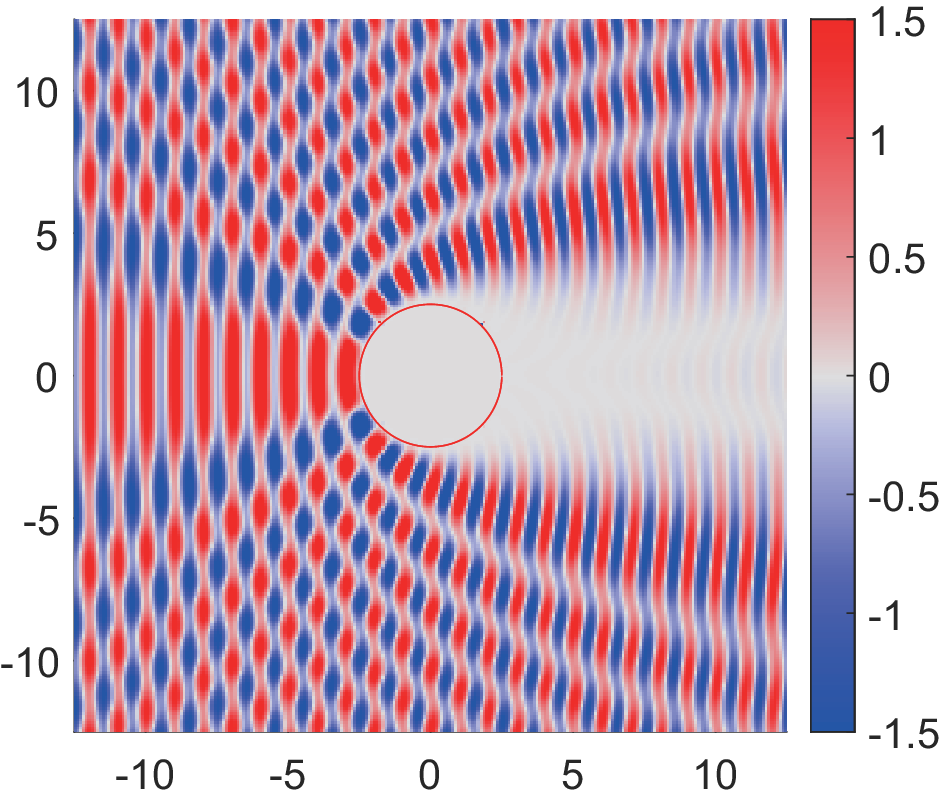}
    \caption{$\mathbf{Re}\left(\mathbb{M}^2[u]\right)$}
    \label{Rand2d}
    \end{subfigure}
    \hfill
    \begin{subfigure}[b]{0.32\linewidth}
    \centering
    \includegraphics[width=\textwidth]{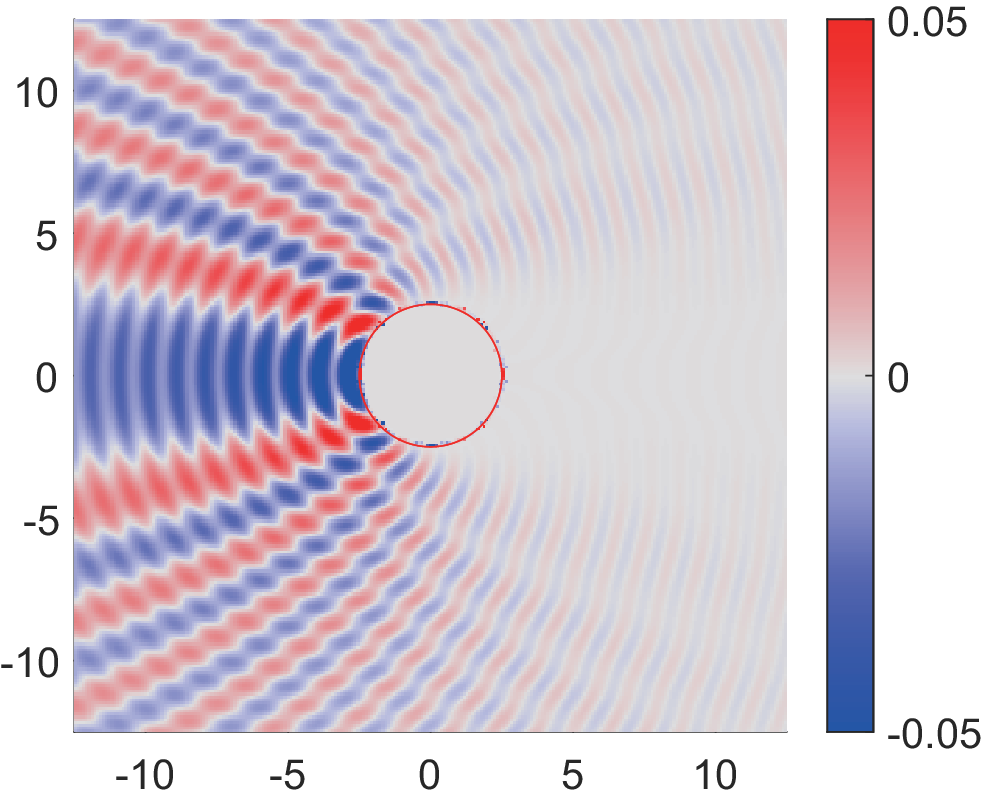}
    \caption{$\mathbf{Re}\left(\mathbb{M}^2[u]- \mathcal{E}^2_1\right)$}
    \label{Rand2e}
    \end{subfigure}
    \hfill
    \begin{subfigure}[b]{0.32\linewidth}
    \centering
    \includegraphics[width=\textwidth]{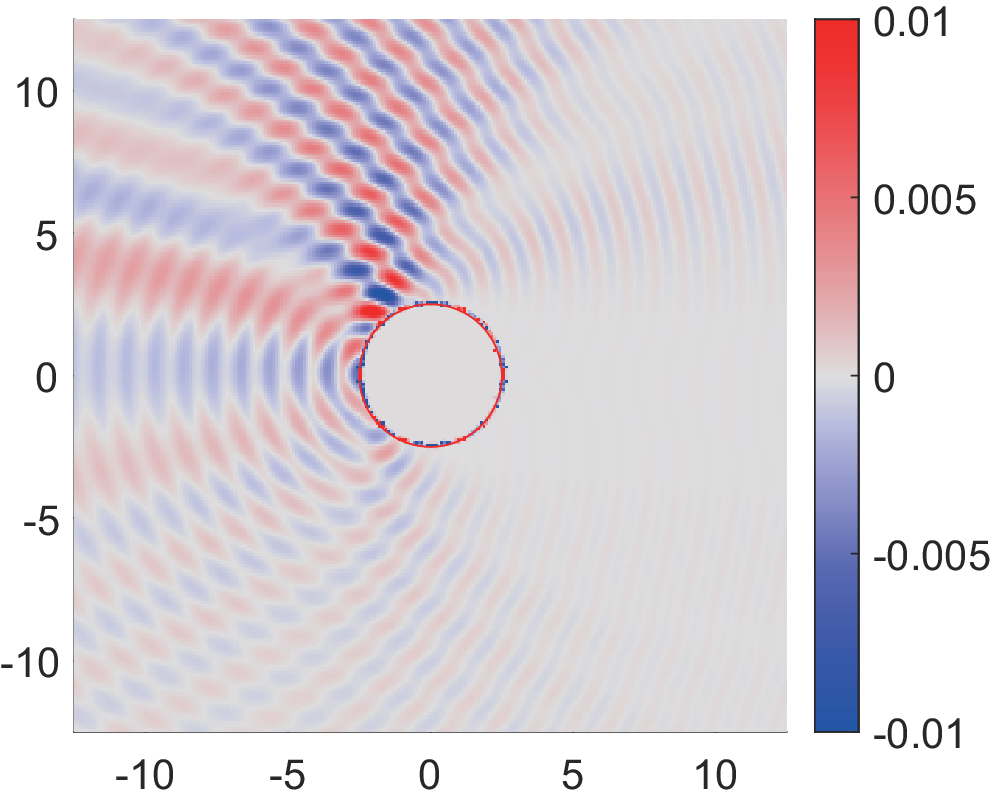}
    \caption{$\mathbf{Re}\left(\mathbb{M}^2[u] - \mathcal{E}^2_2\right)$}
    \label{Rand2f}
    \end{subfigure}

    \caption{The result of Monte Carlo simulations and the estimation error for the shape Taylor expansions. Figures~\ref{Rand2a} and~\ref{Rand2d} show the Monte Carlo simulation of the first and second order moment, respectively. Figures~\ref{Rand2b} and~\ref{Rand2e} show the estimation errors based on the first order shape Taylor expansion. Figures~\ref{Rand2c} and~\ref{Rand2f} show the estimation errors based on the second order expansion.} 
    \label{Rand2}
\end{figure}

During the test, the impedance coefficient in equation~\eqref{scattering_bound2} is set as $\lambda = 100$. The incident field is generated by a plane wave with $\mathbf{z} = [1,0]^\top$ and $k = \pi$. Figure~\ref{Rand2} shows the estimation of the first and second order moments in $D_{\rm ex}$ by applying perturbations with $\epsilon = 0.03$. This result implies that introducing the second order shape derivatives significantly improves the estimation accuracy, especially for the second order moment. In Figure~\ref{Rand3}, we give the estimation residuals with the 1st, 2nd, 4th, and 7th order moments and different perturbation magnitudes $\epsilon$. The results of this experiment also demonstrate that the estimation based on the second order expansion is more reliable, especially when the perturbation amplitude is relatively large. 

According to Theorem \ref{centralmom}, for estimating the variance $\mathbb{VAR}[u]$~\cite{hao2018computation,harbrecht2013first}, the accuracies based on the first and second order shape Taylor expansions are both on the order of $\mathcal{O}\left(\epsilon^4\right)$. In other words, computing the second order expansion does not provide a higher approximation accuracy than the first order expansion. Thus we skip the comparison as a distinct difference can only be observed when using higher order expansions.

\begin{figure}
    \centering
    \includegraphics[width=0.97\linewidth]{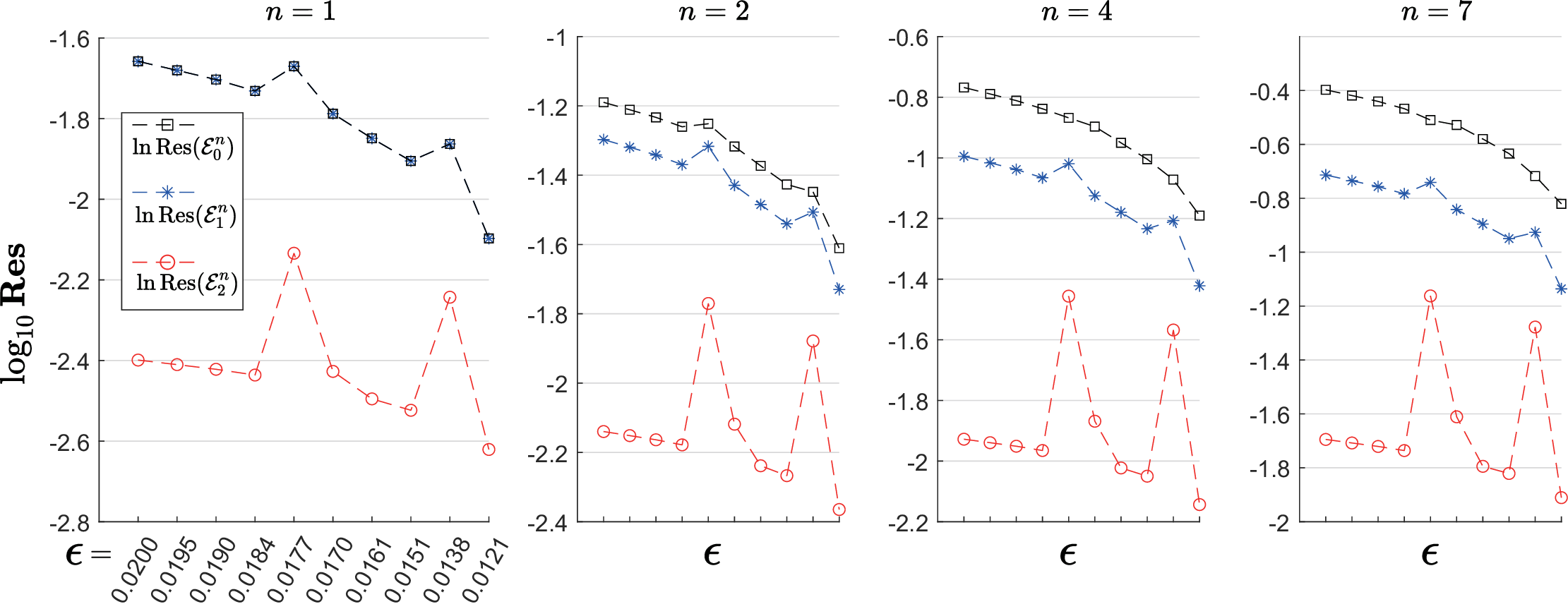}
    \caption{Residuals for the 1st, 2nd, 4th, and 7th moments with different perturbation magnitude. Each subplot contains three residuals corresponding to the 0th, 1st, and 2nd order shape Taylor expansions.}
    \label{Rand3}
\end{figure}

\section{Conclusion}\label{Conclusion}
This paper proposes a computational framework for the shape Taylor expansions in two dimensional acoustic scattering problems with sound-soft, sound-hard, impedance, and transmission boundary conditions. Using boundary integral equations, we discuss the approximation properties of shape Taylor expansions of different orders. We also apply the shape Taylor expansion to uncertainty quantification problems. From the numerical examples, we observe that the higher order shape Taylor expansions offer improved approximations for larger shape perturbations of scatterers and provide efficient and accurate moment estimates for random boundary scattering problems. The computational method can be extended to three-dimensional acoustic scattering, as well as electromagnetic scattering problems. We will explore these extensions in future work.

\bibliographystyle{plain}
\bibliography{REF.bib}

\end{document}